\documentclass{amsart}
\usepackage{a4wide,enumerate,xcolor}
\usepackage{amsmath,graphicx}
\allowdisplaybreaks

\usepackage[colorlinks=true,linkcolor=blue,
anchorcolor=black,citecolor=cyan,filecolor=black,menucolor=black,
runcolor=black,urlcolor=blue]{hyperref}

\usepackage{bbm}
\usepackage{graphicx}        
\usepackage{multicol}        
\usepackage[bottom]{footmisc}
\usepackage{amsmath, amsfonts, amssymb}
\usepackage{color}
\usepackage{pgf, tikz,fancyvrb}
\usepackage{pgfplots}
\usepackage{soul}
\usepackage{todonotes}

\usepackage{lineno}
\modulolinenumbers[5]

\usepackage{subfig}

\newtheorem{theo}{Theorem}[section]
\newtheorem{prop}[theo]{Proposition}
\newtheorem{cor}[theo]{Corollary}
\newtheorem{lem}[theo]{Lemma}
\newtheorem{rema}[theo]{Remark}
\newtheorem{defi}[theo]{Definition}

\def\be{\begin{equation}}
\def\ee{\end{equation}}
\def\sig{\sigma}

\def\RR{\mathbb{R}}
\def\R{\mathbb{R}}
\def\Rplus{\R_{\geq 0}}
\def\Z{\mathbb{Z}}
\def\diver{\nabla\cdot}
\def\grad{\nabla}
\def\Tt{\mathcal{T}}
\def\Ee{\mathcal{E}}
\def\Hh{\mathcal{H}}
\def\Dd{\mathcal{D}}
\def\Aa{\mathcal{A}}
\def\Bb{\mathcal{B}}
\def\Oo{\mathcal{O}}
\def\d{\mathrm{d}}
\def\O{\Omega}
\def\B{\mathfrak{B}}
\def\M{\mathfrak{M}}
\def\Rf{\mathfrak{R}}

\def\Rr{\mathcal{R}}
\def\Ss{\mathcal{S}}
\def\bmu{\boldsymbol{\mu}}
\def\bxi{\boldsymbol{\xi}}

\pgfplotsset{every axis plot/.append style={line width=1pt}}

\newcommand{\logLogSlopeTriangle}[5]
{
	\pgfplotsextra
	{
		\pgfkeysgetvalue{/pgfplots/xmin}{\xmin}
		\pgfkeysgetvalue{/pgfplots/xmax}{\xmax}
		\pgfkeysgetvalue{/pgfplots/ymin}{\ymin}
		\pgfkeysgetvalue{/pgfplots/ymax}{\ymax}
		
		\pgfmathsetmacro{\xArel}{#1}
		\pgfmathsetmacro{\yArel}{#3}
		\pgfmathsetmacro{\xBrel}{#1-#2}
		\pgfmathsetmacro{\yBrel}{\yArel}
		\pgfmathsetmacro{\xCrel}{\xArel}
		
		\pgfmathsetmacro{\lnxB}{\xmin*(1-(#1-#2))+\xmax*(#1-#2)} 
		\pgfmathsetmacro{\lnxA}{\xmin*(1-#1)+\xmax*#1} 
		\pgfmathsetmacro{\lnyA}{\ymin*(1-#3)+\ymax*#3} 
		\pgfmathsetmacro{\lnyC}{\lnyA-#4*(\lnxA-\lnxB)}
		\pgfmathsetmacro{\yCrel}{\lnyC-\ymin)/(\ymax-\ymin)}
		
		\coordinate (A) at (rel axis cs:\xArel,\yArel);
		\coordinate (B) at (rel axis cs:\xBrel,\yBrel);
		\coordinate (C) at (rel axis cs:\xCrel,\yCrel);
		
		\draw[#5]   (A)-- node[pos=0.5,anchor=north] {\scriptsize{1}}
		(B)--
		(C)-- node[pos=0.,anchor=east] {\scriptsize{#4}} 
		(A);
	}
}

\pgfplotsset{compat=1.17} 
\begin{document}

\title[Finite volumes for a Poisson--Nernst--Planck system with cross-diffusion]{Convergence and long-time behavior of finite volumes \\for a generalized Poisson--Nernst--Planck system\\ with cross-diffusion and size exclusion}

\author[C. Canc\`es]{Cl\'ement Canc\`es}
\address{Univ. Lille, CNRS, Inria, UMR 8524 - Laboratoire Paul Painlev\'e, F-59000 Lille, France}
\email{clement.cances@inria.fr, maxime.herda@inria.fr}

\author[M. Herda]{Maxime Herda}

\author[A. Massimini]{Annamaria Massimini}
\address{CERMICS, Ecole des Ponts Paris-Tech and Inria MATHERIALS team-project, 6 et 8 av. Blaise Pascal
Cité Descartes 77455 Marne-La-Vallée, France }
\email{annamaria.massimini@enpc.fr}

\date{\today}

\begin{abstract}
We present a finite volume scheme for modeling the diffusion of charged particles, specifically ions, in constrained geometries using a degenerate Poisson--Nernst--Planck system with size exclusion yielding cross-diffusion. Our method utilizes a two-point flux approximation and is part of the exponentially fitted scheme framework. 
The scheme is shown to be thermodynamically consistent, as it ensures the decay of some discrete version of the free energy. 
Classical numerical analysis results -- existence of a discrete solution, convergence of the scheme as the grid size and the time step go to $0$ --  follow. We also investigate the long-time behavior of the scheme, both from a theoretical and numerical point of view. 
Numerical simulations confirm our findings, but also point out some possibly very slow convergence towards equilibrium of the system under consideration. 

\end{abstract}

\keywords{Drift-diffusion, cross-diffusion, exponential fitting, free energy decay.}

\subjclass[2010]{65M08, 65M12, 35K51.}

\maketitle

\section{Presentation of the problem}
\subsection{The continuous generalized Poisson--Nernst--Planck model }
Motivated by the transfer of ions in confined geometries, Burger {\em et al.} introduced in  \cite{BDPS10} a model accounting for cross-diffusion and size-exclusion effects. The model \cite{BSW12} further incorporated self-consistent electric interaction. 
In this model, $I$ species, the volume fractions of which being denoted by $U = (u_i)_{1\leq i\leq I}$, are subject to diffusion as well as to electric forces induced by a self-consistent electrostatic potential.
Denote by $\O\subset \RR^d$ a bounded connected polyhedral domain, then the conservation of the volume occupied by the species $i$ is written
\begin{equation}\label{eq:contmodel_conserv}
\partial_t u_i + \diver F_i = 0, \qquad i =1,\dots,I, 
\end{equation}
with the flux of the species $i$ being (formally) given by
\begin{equation}\label{eq:Fi}
F_i = - D_i \left(u_0 \grad u_i - u_i \grad u_0 + u_0u_i z_i \grad \phi\right) =  - D_i u_i u_0 \nabla\left(\log\left(\frac{u_i}{u_0}\right)+ z_i\phi\right).
\end{equation}
In the above expression, $D_i>0$ denotes the diffusion coefficient of the species $i$.
The quantity 
\begin{equation}\label{eq:contmodel_compat}
u_0 = 1-\sum_{i=1}^I u_i
\end{equation} 
shall be thought of as the volume fraction of available space for the ions, possibly occupied by a motile and electro-neutral solvent. The quantity $u_0$ is then required to remain non-negative, leading to size exclusion for the other species $u_i$, $i = 1,\dots,I$. Denoting by $z_i \in \Z$ the charge of species $i$ and by $\lambda>0$ the (scaled) Debye length, then the electrostatic potential $\phi$ solves the Poisson equation
\begin{equation}\label{eq:contmodel_Poisson}
-\lambda^2\Delta\phi = \sum_{i=1}^I z_iu_i + f
\end{equation}
for some prescribed {background charge density} $f$.
We consider boundary conditions of mixed type for the electric potential. More precisely, we assume that the boundary $\partial \O$ of the domain can be split into an insulator part $\Gamma^N$, which is a relatively open $C^1$ part of the boundary, and its complement $\Gamma^D$ on which a Dirichlet boundary condition is imposed:
\begin{equation}\label{eq:BC.phi}
\grad \phi \cdot n = 0 \quad \text{on}\; \Gamma^N \qquad \text{and}\quad 
\phi = \phi^D\;\text{on}\; \Gamma^D,
\end{equation}
where $n$ is the unit normal to $\partial \O$.
Throughout this paper, we will assume that $f \in L^\infty(\O)$ and that {$\phi^D$ is the trace of an $L^\infty\cap H^{1}(\O)$ function (which we also denote by $\phi^D$).
 Neither $f$ nor $\phi^D$ depend on time}. 
Boundary conditions of various types can be considered for the conservation laws \eqref{eq:contmodel_conserv}--\eqref{eq:Fi}, for instance a Robin type boundary condition modeling electrochemical reaction thanks to a Butler--Volmer type formula, see for instance~\cite{CCMRV23}, or boundary conditions of mixed Dirichlet--Neumann type as in~\cite{GJ18}. 
In the presentation of the scheme, we assume for simplicity that the system is isolated, in the sense that 
\begin{equation}\label{eq:BC.ui}
F_i \cdot n = 0 \quad \text{on}\; \partial \O, \qquad i = 1,\dots, I. 
\end{equation}
The system is finally complemented with initial conditions $u_i(t=0) = u_i^0$ with
\begin{equation}\label{eq:init}
\text{
$u_i^0\geq 0$ \quad \text{and}\quad $\int_\O u_i^0 >0$\quad for\quad $i=0,\dots,I$ \quad and \quad$\sum_{i=0}^I u_i^0 = 1$.
}
\end{equation}
We then denote by 
\[
\Aa = \left\{ U = \left( u_i \right)_{1 \leq i \leq I}\in (\Rplus)^I \; \middle| \; \sum_{i=1}^I u_i \leq 1\right\}
\]
the set in which the volume fractions have to take their values. 
\subsection{Entropy structure of the model}

Let us now describe the entropy (or formal gradient flow) structure of the model. 
Introduce the Slotboom variables $w_i = \frac{u_i}{u_0}e^{z_i\phi}$, then the fluxes~\eqref{eq:Fi} rewrite as 
\be\label{eq:Fi.Slotboom}
F_i = - D_i u_0^2 e^{-z_i \phi} \grad w_i, \qquad i = 1,\dots, I. 
\ee
If the so-called electrochemical potentials defined by $\mu_i = \log w_i = \log \frac{u_i}{u_0} + z_i \phi$ are regular enough, say in $H^1\cap L^\infty(\O)$, then 
\(
\grad \log w_i \cdot \grad w_i = 4 | \grad \sqrt{w_i}|^2.
\)
Therefore, multiplying~\eqref{eq:contmodel_conserv} by $\mu_i$,
 integrating over $\O$ and summing over $i = 1,\dots, I$ then yields 
\be\label{eq:EDE}
\frac{\d}{\d t} \Hh + 4 \int_\O \sum_{i = 1}^I D_i u_0^2 e^{-z_i \phi} | \grad \sqrt{w_i} |^2 = 0,  
\ee
where, denoting the mixing (neg)entropy density function $H: \Aa \to [-\log(I+1),+\infty)$ by 
\be\label{eq:H}
H(U) =  u_0 \log(u_0) + \sum_{i=1}^I u_i \log(u_i), 
\ee
with $u_0$ seen as a function of $U$, the free energy $\Hh$ is given by 
\be\label{eq:H_tot}
\Hh = \int_\O H(U) + \frac{\lambda^2}2  \int_\O |\grad \phi|^2 - \lambda^2 \int_{\Gamma^D}\phi^D \grad \phi\cdot n.
\ee
Assume that the $u_i$ are positive for $i = 0,\dots,I$ (as proved in the discrete case later on), then the second term in \eqref{eq:EDE} is well-defined and non-negative. 
As a consequence, the free energy decays along-time, as a manifestation of the second principle of thermodynamics. {Observe that $\Hh$ need not be non-negative but may be bounded uniformly from below by a constant depending only on $\lambda$, $f$, and $\phi^D$.
}

\subsection{Weak solution}\label{ssec:weak}
As we aim to prove rigorously the convergence of the scheme to be presented in Section~\ref{sec:schemes}, we need a proper notion of solution for the continuous problem~\eqref{eq:contmodel_conserv}--\eqref{eq:init}.
The volume fractions we seek are such that 
\be\label{eq:cont.ui.Linf}
\text{$0 \leq u_i \leq 1$ for $i=0,\dots,I$ and a.e. $(t,x) \in \Rplus \times \O$.}
 \ee
The electric potential $\phi$ then solves the Poisson equation~\eqref{eq:contmodel_Poisson} with a bounded right-hand side and boundary condition, so it satisfies
\be\label{eq:cont.phi.Linf+H1}
\| \phi \|_{L^{\infty}(\Rplus \times \O)} + \| \phi -\phi^D\|_{L^{\infty}(\Rplus; V)}\leq C(\lambda, (z_i)_i, \O, f), 
\ee
where 
\[
V = \left\{ v \in H^1(\O) \; \middle| \; v_{|_{\Gamma^D}} = 0\right\}
\]
is equipped with the $H^1(\O)$ semi-norm $\| \grad v \|_{L^2}$, which is a norm thanks to the Poincaré inequality. 
We have emphasized in the right-hand side of~\eqref{eq:cont.phi.Linf+H1} the dependence of the bound on the data, especially on the Debye length $\lambda$.
The estimate~\eqref{eq:EDE} is the other main estimate on which our study will rely. It provides enough control 
to define a proper notion of weak solutions, but at the price of a suitable reformulation of the fluxes~\eqref{eq:Fi}. 
Let us first remark that, since the free energy $\Hh$ is bounded (see Lemma~\ref{lem:boundentropy}), integrating
\eqref{eq:EDE} over $t \in \Rplus$ provides 
\be\label{eq:cont.dissip}
 \iint_{\Rplus\times\O} \sum_{i = 1}^I D_i u_0^2 e^{-z_i \phi} | \grad \sqrt{w_i} |^2 \leq C. 
\ee
As already noticed in \cite{GJ18}, the above inequality yields a control in $L^2_{\rm loc}(\Rplus; H^1(\O))$ on $\sqrt{u_0}$ and on $u_0$ itself, as well as control on some product terms. More precisely, in the numerical analysis exposed hereafter, we derive a $L^2_{\rm loc}(\Rplus; H^1(\O))$ estimate on $u_i\sqrt{u_0}$, which, together with~\eqref{eq:cont.ui.Linf} provides a $L^2_{\rm loc}(\Rplus; H^1(\O))$ control on $u_i u_0$ too. Hence, all the terms in the following expression of the fluxes
\be\label{eq:Fi.3}
F_i = - D_i \left( \grad (u_0 u_i) - 4 u_i \sqrt{u_0}  \grad \sqrt{u_0} + u_i u_0 z_i \grad \phi \right), \qquad 1 \leq i \leq I, 
\ee
have a clear mathematical sense, motivating the following notion of weak solution. 

\begin{defi}[Weak solution]\label{def:weak}
$(U,\phi)$ is said to be a weak solution to~\eqref{eq:contmodel_conserv}--\eqref{eq:init} if 
\begin{itemize}
\item $u_i \in L^\infty(\Rplus\times\O)$ with $U \in \Aa$ a.e. in $\Rplus\times\O$, with moreover $\sqrt{u_0}$ and $u_i\sqrt{u_0}$ belonging to $L^2_{\rm loc}(\Rplus; H^1(\O))$ for $u_0$ defined as in~\eqref{eq:contmodel_compat}; 
\item $\phi \in L^\infty(\Rplus\times \O)$ with $\phi - \phi^D \in L^{\infty}(\Rplus; V)$ satisfies 
\be\label{eq:weak.phi}
\int_\O \lambda^2 \grad \phi(t,x) \cdot \grad \psi(x) \d x = \int_\O \left( \sum_{i=1}^I z_i u_i(t,x) + f(x) \right) \psi(x) \d x
\ee
for all $\psi \in C^1_c(\O \cup \Gamma^\mathrm{N})$ and a.e. $t \ge 0$; 
\item for all $\varphi \in C^1_c(\Rplus \times \overline \O)$ and all $i=1,\dots, I$, there holds 
\be\label{eq:weak.ui}
\iint_{\Rplus \times \O} u_i \partial_t \varphi + \int_\O u_i^0 \varphi(0,\cdot) - \iint_{\Rplus \times \O}  D_i
\big( \grad (u_0 u_i) - 4 u_i \sqrt{u_0} \grad \sqrt{u_0} + u_i u_0 z_i \grad \phi  \big) \cdot \grad \varphi = 0. 
\ee
\end{itemize}
\end{defi}

\subsection{Equilibrium and modified Poisson--Boltzmann equation}\label{ssec:equilibrium}
As the model has a gradient flow structure and is equipped with no-flux boundary conditions, its solution converges 
towards a steady state $(u^\infty, \phi^\infty)$ which correspond to vanishing fluxes $F_i^\infty$ for all $i=1,\dots, I$. 
It is suggested in \cite{BSW12} that the long-time limit profile $(u^\infty, \phi^\infty)$ should correspond to constant in 
space electrochemical potentials $\boldsymbol{\mu}^\infty = \left(\mu_i^\infty\right)_{1\leq i \leq I}$. Define the smooth functions $v_i: \R \times \R^I \to (0,1)$ by 
\be\label{eq:vi}
v_i\left(y, \boldsymbol{\xi} \right) =  \frac{e^{\xi_i - z_i y}}{1+ \sum_{j=1}^I e^{\xi_j - z_j y}} \quad \text{for $1\leq i \leq I$}, 
\quad \text{and} \quad v_0\left(y, \boldsymbol{\xi} \right) = 1- \sum_{j=1}^I v_j(y,\boldsymbol{\xi}), 
\ee
where $\boldsymbol{\xi} = \left(\xi_j\right)_{1 \leq j \leq I} \in \R^I$.
Then the volume fractions $U^\infty = \left(u_i^\infty\right)_{1\leq i \leq I}$ satisfy 
\be\label{eq:cont.uiinf}
u_i^\infty(x) = v_i(\phi^\infty(x),\boldsymbol{\mu}^\infty), 
\qquad i = 0,\dots, I, \; 
x \in \O.
\ee
Plugging these expressions in the Poisson equation~\eqref{eq:contmodel_Poisson} provides what we call the \textit{modified Poisson--Boltzmann equation}, i.e. 
\be\label{eq:cont.PoissonBoltzmann}
- \lambda^2 \Delta \phi^\infty + r(\phi^\infty, \boldsymbol{\mu}^\infty) = f \quad \text{in}\; \O, 
\ee
complemented with the same boundary conditions~\eqref{eq:BC.phi}. In~\eqref{eq:cont.PoissonBoltzmann}, $r$ is the smooth function from $\R \times \R^I$ to $\R$ defined by 
\be\label{eq:r}
r(y, \boldsymbol{\xi}) = -  \frac{ \sum_{i=1}^I z_i e^{\xi_i - z_i y }}{1+ \sum_{i=1}^I e^{\xi_i - z_i y }} = - \sum_{i=1}^I z_i v_i(y,\bxi). 
\ee
To close the system, we still need to determine the real constants $\boldsymbol{\mu}^\infty$. Their values are calculated to guarantee the conservation-of-total-mass constraints 
\be\label{eq:cont.steadymass}
\int_\O u_i^\infty = \int_\O u_i^0.
\ee
By setting $z_0 = 0$, then one readily checks that since the $(v_i)$ by convexity 
\[
\partial_y r(y, \boldsymbol{\xi}) = \sum_{i=0}^I z_i^2 v_i(y, \boldsymbol{\xi}) -  \left( \sum_{i=0}^I z_i v_i(y, \boldsymbol{\xi}) \right)^2 \geq 0, 
\]
the last inequality being a consequence of the convexity of $y \mapsto y^2$ as $ \sum_{i=0}^I v_i(y, \boldsymbol{\xi}) =1$. 
Therefore, $r$ is non-decreasing w.r.t.~its first variable $y$. The monotonicity and boundedness of $r$ allow us to conclude, via classical monotonicity arguments, the well-posedness of the problem~\eqref{eq:cont.PoissonBoltzmann}\&\eqref{eq:BC.phi}. 
Going one step further allows us to check that the solution $(\phi^\infty, \boldsymbol \mu^\infty)$ is 
a minimizer of the functional $\Psi: (\phi^D + V) \times \R^I \to \R$ defined by 
\be\label{eq:steady.Psi}
\Psi(y,\boldsymbol{\xi}) =\int_\O  \left[ \frac{\lambda^2}2 |\grad y|^2 + \log\left( 1+ \sum_{i=1}^I e^{\xi_i - z_i y} \right) - f y - \sum_{i=1}^I \xi_i u_i^0 \right].
\ee
One readily checks that $\Psi$ is strictly convex, and thus the minimiser $(\phi^\infty, \boldsymbol \mu^\infty)$ is unique.  
Moreover, one easily verifies that minimizers for $\Psi$ are steady states of the system~\eqref{eq:contmodel_conserv}--\eqref{eq:init}. 
Despite hints in these directions, the uniqueness of the steady state is neither established in~\cite{BSW12}, nor, to our knowledge, in more recent contributions. Furthermore, no quantitative 
convergence result has been established so far. Non-quantitative results were however recently presented in~\cite{CJLL24} in the absence of 
electrical interactions, i.e. $z_i=0$ for all $i$.
We partially fill this gap in the discrete case. We rigorously establish in Section~\ref{sec:asymptotic} the uniqueness of the steady states, which are minimisers of some discrete counterpart of the functional $\Psi$,  cf.~\eqref{eq:PsiTt}, and the (non-quantitative) convergence towards such a steady profile as time goes to $+\infty$.

\subsection{Goal and positioning of the paper} 
{The goal of this paper is to provide a developed mathematical study of a numerical scheme that was proposed in our previous contribution~\cite{CHM_FVCA}. It enters the broad family of the two-point flux approximation (TPFA) finite volume schemes~\cite{Eymard2000}, which has been applied to several cross-diffusion problems in recent years (see e.g. \cite{CG20, CFS20, JZ21, GF22, CZ22, JZ23, CEM24}). Our scheme mixes ideas from the so-called square-root approximation (SQRA) finite volume scheme~\cite{CV23} (see also~\cite{LFW13, Heida18}) with some features of the Scharfetter--Gummel scheme \cite{SG69, Chatard_FVCA6}, especially the use of the Bernoulli function~\eqref{eq:B.SG} in the definition of the fluxes. Note that a purely SQRA approach was also proposed in~\cite{CHM_FVCA}, but since it is overpassed by the mixed approach both in terms of accuracy and of robustness in the small Debye length (or quasi-neutral) limit, we do not develop the analysis here. Let us however stress that our whole study extends to this full-SQRA scheme at the price of very minor modifications sketched in~\cite{CHM_FVCA}.}

{As highlighted in~\cite{CHM_FVCA}, our scheme fulfils key properties. First, it is second order accurate in space, in opposition to the scheme proposed in~\cite{CCGJ19} in which upwinding leads to a mere first order consistency. Even though our scheme does not use the so-called entropy variables (the electrochemical potentials in our framework) as unknowns in opposition to the finite element scheme presented in~\cite{GJ19}, our scheme satisfies a discrete energy dissipation inequality, which should be thought of as a discrete counterpart to~\eqref{eq:EDE}, and which is key for its numerical analysis. Keeping conservative variables as primary variables yields well-behaved non-linear solvers (we used Newton--Raphson), allowing us to consider larger time steps than for the finite element methods of~\cite{GJ19} or for the finite volume method of~\cite{BCH23}. Together with the stability to be used for proving the convergence of the scheme, the discrete energy dissipation inequality implies that the scheme preserves exactly the steady states of the continuous model in the sense that, at the discrete level, fluxes vanish when concentrations are of the form~\eqref{eq:cont.uiinf}. The unique steady state of our scheme is
then characterized only by the electric potential which solves a discrete counterpart of the
Poisson--Boltzmann equation~\eqref{eq:cont.PoissonBoltzmann}, which can be computed directly in an efficient way. }

{In the present contribution, we first rigorously establish the convergence of the scheme when the discretization parameters (mesh size and time step) tend to $0$. As in the continuous setting~\cite{GJ18}, the proof is based on compactness arguments. It strongly relies on the discrete counterpart of the energy dissipation estimate~\eqref{eq:EDE}, following the boundedness-by-entropy method~\cite{J15}, and more precisely a discrete version of it~\cite{CH14_FVCA7, JZ23}. Uniqueness of the weak solutions is moreover established in~\cite{GJ18} under moderate regularity assumptions thanks to the uniqueness technique of Gajewski~\cite{G94, GS03}.}

{
We also investigate the long-time behaviour of the numerical scheme. We show in particular  that the discrete solution to our scheme for the evolution problem converges as time goes to infinity towards the unique solution to the discrete steady problem. As the proof relies on LaSalle's invariance principle, no quantitative convergence rate is derived, similarly to what is proposed in~\cite{CJLL24} in the continuous setting but in the absence of electrical interactions. By establishing this result, we rigorously establish in the discrete setting the long-time behavior conjectured in~\cite{BSW12} and recalled in Section~\ref{ssec:equilibrium}}. Numerical simulations suggest that the convergence rate strongly depends on the Debye length, but also on the initial profile $U^0$ (in opposition to the classical Poisson--Nernst--Planck problem without volume filling~\cite{BCCH17}).

\section{The finite volume scheme and main results}\label{sec:schemes}
First, we introduce the time discretization and the spatial mesh of the domain $\O$.

\subsection{Space and time discretizations} 
\label{sec:space_time_discr}
The mesh will be assumed to be admissible in the sense of \cite{Eymard2000}, namely it fulfils the so-called {\em orthogonality condition}, which is usual for two-point flux approximation
finite volumes.

    Let $\mathcal{T}$ denote a family of non-empty, disjoint, convex, open, and polygonal \textit{control volumes} $K \in \mathcal{T}$, whose Lebesgue measure is denoted by $m_K$. 
    We also assume that control volumes partition the domain, meaning that $\overline{\O} = \bigcup_{K \in \mathcal{T}} \overline{K}$. 
    Further, we call $\mathcal{E}$ a \textit{family of edges/faces}, where each $\sigma \in \mathcal{E}$ is a closed subset of $\overline{\O}$ contained in a hyperplane of $\mathbb{R}^d$. Each $\sigma$ has a strictly positive $(d-1)$-dimensional Hausdorff (or Lebesgue) measure, denoted by $m_\sigma$.  We use the abbreviation $K | L = \partial K \cap \partial L$ for the intersection between two distinct control volumes which is either empty or reduces to a face contained in $\mathcal{E}$. The subset of all interior faces is denoted by 
    \[\mathcal{E}_{\rm int} = \left\{\sigma\in\mathcal{E} \; \middle| \;\sigma = K|L\text{ for some }K,L\in\mathcal{T}\right\}.\] 
For any $K \in \mathcal{T}$, we assume that there exists a subset $ \mathcal{E}_K$ of distinct elements of $\mathcal{E}$ such that the boundary of a control volume can be described by $\partial K = \bigcup_{\sigma \in \mathcal{E}_K} \sigma$ and, consequently, $\mathcal{E} = \bigcup_{K \in \mathcal{T}} \mathcal{E}_K$. Additionally, we assume that boundary edges $\Ee_\text{ext} = \Ee \setminus \Ee_{\rm int}$ 
are either subsets of $\Gamma^D$ or $\Gamma^N$. We denote by $\Ee^D, \Ee^N \subset \Ee$ the sets of boundary edges lying on $\Gamma^D$ and $\Gamma^N$, respectively.
    To each control volume $K \in \mathcal{T}$ we assign a \textit{cell center} ${x}_K \in K$ which satisfies the \textit{orthogonality condition}: If $K, L$ share a face $\sigma =K | L$, then the vector {\color{black}$x_L - x_K$ is orthogonal to $\sigma = K | L$}.
    The triplet $\left( \mathcal{T}, \mathcal{E}, \lbrace{x}_K\rbrace_{K \in \mathcal{T}}\right)$ is called an \emph{admissible mesh}.

   A classical way to construct such admissible mesh is to place some seeds $(x_K)_{K\in\Tt}$ in the domain $\O$ and to define 
   the corresponding Vorono\"\i\; tessellation. As in practice the condition $x_K \in K$ can be relaxed, but sill with the constraint $x_L-x_K$ points outward with respect to $K$, another popular choice to build admissible meshes from a simplicial mesh satisfying the Delaunay condition is to choose $x_K$ as the center of the circumcircle of the simplex $K$.

    We introduce the notation $d_\sigma$ for the Euclidean distance between ${x}_K$ and ${x}_L$ if $\sigma= K|L$ or between ${x}_K$ and the affine hyperplane spanned by $\sigma$ if $\sigma\subset\partial{\O}$. We also denote by $d_{K\sig} = \operatorname{dist}(x_K,\sig)$, so that $d_\sig = d_{K\sig}+d_{L\sig}$ if $\sig = K|L\in \Ee_{\rm int}$ and $d_\sig = d_{K\sig}$ if $\sig \in \Ee_K \cap \Ee_\text{ext}$.
The {\em transmittivity} of the edge $\sig \in \Ee$ is defined by ${a_\sig = \frac{m_\sig}{d_\sig}}$.
The size of the mesh is
 \[h_\Tt = \max_{K\in\mathcal{T}}\operatorname{diam}(K)\] 
where $\operatorname{diam}(K)$ denotes the diameter of the cell $K$. The regularity of the mesh is defined by 
\[
\zeta_\Tt = \max_{K\in\Tt} \left( \operatorname{card}\, \Ee_K \; ; \; \max_{\sig \in \Ee_K} \frac{\operatorname{diam}(K)}{d_{K\sig}}
\right).
\]

    {For the time discretization, we decompose the time interval $\mathbb{R}_{\geq 0}:=[0, +\infty)$ into an unbounded increasing sequence $\left(t^n\right)_{n\geq 0}$ with $t^0 = 0$ and possibly non-uniform time steps
    \[\tau^n = t^{n} - t^{n-1} >0, \qquad n >0. \] 
    We finally introduce $\Delta t = \sup_{n \in\mathbb{N}\setminus\{0\}}\tau^n$, which we assume to be finite.} As we consider the unbounded time interval $\Rplus$,  we furthermore have to assume that 
    \(
    \inf_{n\geq 1} \tau^n >0. 
    \)

\subsection{The finite volume schemes}
We are now in position to define the finite volume scheme.    
Let us start with the discretization of the Poisson equation~\eqref{eq:contmodel_Poisson}--\eqref{eq:BC.phi}, which relies on a classical two-point flux approximation
\be\label{eq:scheme.Poisson}
\lambda^2 \sum_{\sig \in \Ee_K} a_\sig( \phi_K^n - \phi_{K\sig}^n) = m_K\left( f_K + \sum_{i=1}^I z_i u_{i,K}^n \right), \qquad K\in\Tt,
\ee
where $f_K$ is (possibly an approximation of) the mean value of $f$ on the cell $K$, and where 
\[
\phi_{K\sig}^n = \begin{cases}
\phi_L^n & \text{if}\; \sig = K|L \in \Ee_{\rm int}, \\
{\phi_K^n} & {\text{if}\; \sig \subset \Gamma^N,}\\  
\phi_\sig^D = \frac1{m_\sig}\int_\sig \phi^D & \text{if}\; \sig \subset \Gamma^D.  
\end{cases}
\]    
The equation~\eqref{eq:contmodel_conserv} is discretized using a backward Euler method in time and finite volumes in space, leading to
    \begin{equation}\label{eq:fvscheme_conserv}
    \frac{u_{i,K}^{n} - u_{i,K}^{n-1}}{\tau^n}m_K + \sum_{\sigma\in\Ee_K} F_{i,K\sigma}^n = 0, \quad i = 1,\dots, I, \quad K \in \Tt.
    \end{equation}
In accordance with~\eqref{eq:BC.ui}, we set $F_{i,K\sig}^n = 0$ if $\sig \subset \partial\O$. For $\sig = K|L$ an internal edge, then we define 
    \begin{equation}\label{eq:scheme.Fi}
    F_{i,K\sigma}^n = a_\sig D_i \left(u^n_{i,K}u^{n}_{0,L} \B\left(z_i(\phi_L^n - \phi_K^n)\right)
   - u^n_{i,L}u^{n}_{0,K}\B\left(z_i(\phi^n_{K}-\phi^n_{L})\right)\right), 
    \end{equation}
    with 
    \begin{equation}\label{eq:u0Kn}
    u_{0,K}^n = 1 - \sum_{i = 1}^I u_{i,K}^n, \qquad K \in \Tt. 
    \end{equation}
Formula~\eqref{eq:scheme.Fi} involves a function $\B\in C^1(\R;\R)$ which is (strictly) positive and satisfies $\B(0) = 1$ and $\B'(0) = -1/2$. The SQRA scheme studied in~\cite{CHM_FVCA} corresponds to the choice $\B(y) = e^{-y/2}$. Even if our analysis extends to this case without any particular difficulty, for the sake of readability we will focus here on the scheme referred to as the SG scheme in~\cite{CHM_FVCA}, for which $\B$ is the Bernoulli function
\be\label{eq:B.SG}
\B(y) = \frac{y}{e^y-1}, 
\ee
and which shows a better behavior (accuracy and robustness) with respect to the SQRA scheme. We refer to~\cite{CHM_FVCA} for a comparison of the two approaches.  

    We would like to emphasize that the construction of the scheme which we will refer to as the SG scheme is not based on the original idea of \cite{SG69}. Rather, we take advantage of the free-energy diminishing character of the SG scheme highlighted in \cite{Chatard_FVCA6} and exploited in \cite{BCV14, SS22}.

In order to close the system, it remains to define the discrete counterpart of $u^0$ as follows:
\be\label{eq:uiK0}
u_{i,K}^0 = \frac1{m_K} \int_K u_i^0, \qquad K\in\Tt, \; i = 0,\dots, I. 
\ee
Then we infer from~\eqref{eq:init} that 
\be\label{eq:uiK0.2}
\sum_{i=0}^I u_{i,K}^0 = 1 \; \text{for all $K\in\Tt$}, \; \text{and}\; \sum_{K\in\Tt} u_{i,K}^0 m_K = \int_\O u_{i}^0 >0 \quad \text{for $i=0,\dots, I$}. 
\ee
In what follows, we denote by $U_K^n = \left(u_{i,K}^n\right)_{i=1,\dots,I}$ for $K\in\Tt$ and $n \geq0$.

The consistency of the discrete fluxes \eqref{eq:scheme.Fi} with the continuous ones \eqref{eq:Fi} might not look completely obvious. 
It follows from the identity
\be\label{eq:scheme.Fi.2}
F_{i,K\sigma}^n = a_\sig\, D_i\, u_{0,K}^n \,u_{0,L}^n \M(e^{-z_i \phi_K^n}, e^{-z_i \phi_L^n})\left(w_{i,K}^n - w_{i,L}^n\right)
\ee
where $w^n_{i,K} = \frac{u^n_{i,K}}{u^n_{0,K}}e^{z_i\phi^n_K}$ are the discrete counterpart of the Slotboom variables and, by assumption, $U^n_K \in (0,1)^I$. Later, in Proposition~\ref{prop:positivity}, this bound will be rigorously established.
The mean function $\M$ appearing in \eqref{eq:scheme.Fi.2} is defined by 
\[
\M(a,b) = \frac{\log(1/a) - \log(1/b)}{1/a - 1/b} \quad \text{for $a,b >0$ with $a\neq b$, and $\M(a,a) = a$ for $a>0$}. 
\]
We refer to~\cite{HKS21} for an extensive discussion on the influence of the choice of the Stolarsky mean $\M$ on the scheme behavior. The consistency of the formula~\eqref{eq:scheme.Fi.2} with the expression \eqref{eq:Fi.Slotboom} of the continuous flux is now clear, and assuming the existence of a regular solution to the continuous problem, inserting it in the scheme and performing Taylor expansions as in~\cite{CV23} shows the second order accuracy in space of our scheme on uniform Cartesian grids. 

\subsection{Main results and organization of the paper}

We provide here simple presentations of our main results, which will be detailed and proven in the next sections. 
As already mentioned, the goals of this contribution are twofold. It first aims to complement the elements of analysis sketched in the contribution~\cite{CHM_FVCA}, by showing existence of discrete solutions and convergence of the scheme. Secondly it intends to investigate the long-time behavior of the scheme.

Our first main result concerns the well-posedness of the scheme.
\begin{theo}\label{theo:main.1}
Given an admissible mesh $(\Tt, \Ee, (x_K)_{K\in\Tt})$ of $\O$ and a sequence of time steps $\left(\tau^n\right)_{n\geq 0}$ as in Section~\ref{sec:space_time_discr}, then there exists (at least) one solution to the scheme~\eqref{eq:scheme.Poisson}--\eqref{eq:uiK0} which satisfies 
\[
0<u_{i,K}^n<1  \qquad \forall i = 0, \dots, I, \; K \in \Tt, \; n \geq 1. 
\]
Moreover, the discrete free energy $\Hh_\Tt^n$ defined later on in~\eqref{eq:free_energ_discr} is decaying along the time iterations
\be\label{eq:disc.EDI}
\Hh_\Tt^n + \tau^n \Dd_\Tt^n  \leq \Hh_\Tt^{n-1}, \qquad n \geq 1, 
\ee
for some dissipation rate $\Dd_\Tt^n\geq 0$ vanishing if and only if 
$\left(\left(U_K^n\right)_{K\in\Tt},\left(\phi_K^n\right)_{K\in\Tt}\right)$ is the stationary solution to the scheme. 
\end{theo}

Theorem~\ref{theo:main.1} states that, at fixed mesh and time step, the numerical scheme admits a positive solution and preserves the $L^\infty$-bounds. Moreover, the inequality~\eqref{eq:disc.EDI} is a discrete counterpart to~\eqref{eq:EDE}. The discrete version $\mathcal{H}_\Tt^n$ of the free energy functional~\eqref{eq:H_tot} is rigorously defined in formula~\eqref{eq:free_energ_discr}. Schemes that verify~\eqref{eq:disc.EDI} in addition to more classical properties such as mass conservation or positivity preservation are often referred to in the literature as thermodynamically consistent or structure-preserving. Our scheme, therefore, falls into this class of schemes.

Theorem~\ref{theo:main.1} allows to define the so-called approximate solution to the problem, which consists of functions $U_{\Tt,\tau} = \left(u_{i,\Tt, \tau}\right)_{1\leq i \leq I}$ and $\phi_{\Tt,\tau}$, piecewise constant in both space and time, defined by 
\be\label{eq:uTt-phiTt}
u_{i,\Tt,\tau}(t,x) = u_{i,K}^n \quad \text{and}\quad  \phi_{\Tt,\tau}(t,x) = \phi_{K}^n \quad \text{if}\; (t,x) \in K \times (t^{n-1}, t^n].
\ee
Here again, we reconstruct $u_{0,\Tt,\tau}$ from $U_{\Tt,\tau}$ by setting 
\be\label{eq:u0Tt}
u_{0,\Tt, \tau} = 1 - \sum_{i=1}^I u_{i,\Tt, \tau}. 
\ee

Now, let $\left( \mathcal{T}_\ell, \mathcal{E}_\ell, \lbrace{x}_K\rbrace_{K \in \mathcal{T}_\ell}\right)_{\ell \geq 1}$ and $\left(\tau_\ell\right)_{\ell\geq 1}=\left(\left(\tau^n_\ell\right)_{n\geq 1}\right)_{\ell\geq 1}$ 
be respectively a sequence of admissible meshes and a sequence of time steps, in the sense of Section~\ref{sec:space_time_discr}, such that
\[\lim_{\ell \to \infty} h_{\Tt_\ell} =\lim_{\ell \to \infty} \Delta t_{\ell}  = 0, \quad \quad \Delta t_{\ell} = \max_n \tau^n_\ell,\]
while the mesh regularity factor $\zeta_{\Tt_\ell}$ remains uniformly bounded w.r.t. $\ell$, i.e. $\zeta_{\Tt_\ell}\leq \zeta_{\star} < + \infty.$
The convergence of our scheme can be stated in a very simplified way as follows:
\begin{theo}\label{theo:main.2}
There exists a weak solution $(U, \phi)$ in the sense of Definition~\ref{def:weak} such that, up to the extraction of a subsequence, 
\[
\phi_{\Tt_\ell, \tau_\ell} \underset{\ell \to + \infty} \longrightarrow \phi \quad\text{and}\quad 
u_{0,\Tt_\ell, \tau_\ell} \underset{\ell \to + \infty} \longrightarrow u_0 \quad \text{in}\; L^p_{\rm loc}(\Rplus; L^p(\O)) \quad \forall p \in [1,+\infty), 
\]
and 
\[
U_{\Tt_\ell, \tau_\ell} \underset{\ell \to + \infty} \longrightarrow U \quad \text{in the } L^\infty(\Rplus\times\O)^I \text{ weak-$\star$ sense}.
\]
\end{theo}
More convergence properties have to be established to prove Theorem~\ref{theo:main.2}. 
They are detailed in Section~\ref{sec:conv}, complementing the preliminary presentation of~\cite{CHM_FVCA}.

Our last main theoretical result is about the long-time asymptotic of the scheme for a fixed mesh and sequence of time steps. We denote by $\Psi_\Tt: \R^\Tt \times \R^I \to \R$ the functional defined by 
\begin{multline}\label{eq:PsiTt}
\Psi_\Tt(y_\Tt ,\boldsymbol \xi) = \frac{\lambda^2}2 \sum_{\sigma \in \Ee} a_\sig (y_K - y_{K\sig})^2 
+  \sum_{K\in\Tt} m_K \log\left( 1 + \sum_{i=1}^I e^{\xi_i - z_i {y}_K} \right)
\\
- \sum_{K\in\Tt} m_K \left[ f_K {y}_K + \sum_{i=1}^I \xi_i u_{i,K}^0 \right]. 
\end{multline}
In~\eqref{eq:PsiTt}, we made a slight abuse of notation as we denote by ${y}_\Tt = \left({y}_K\right)_{K\in\Tt}$ the element of $\R^\Tt$, and the mirror values ${y}_{K\sig}$ are defined by 
\[
{y}_{K\sig} =
\begin{cases}
{y}_L & \text{if}\; \sig = K|L \in \Ee_{\rm int}, \\
{y}_K& \text{if}\; \sig \subset \Gamma^N, \\
\phi_\sig^D &  \text{if}\; \sig \subset \Gamma^D.
\end{cases}
\]
\begin{theo}\label{theo:main.3}
The functional $\Psi_\Tt$ is strictly convex and admits a unique minimizer denoted by $\left(({\phi_K^\infty)}_{K\in\Tt},
{(\mu_{i,\Tt}^\infty)}_{1\leq i \leq I} \right)$. Moreover, one has 
\[
\phi_K^n \underset{n\to +\infty} \longrightarrow \phi_K^\infty \qquad \forall K \in \Tt, 
\]
and 
\[
u_{i,K}^n \underset{n\to +\infty} \longrightarrow v_i(\phi_K^\infty, \boldsymbol{\mu}_\Tt^\infty) \qquad \forall K \in \Tt, \; 0 \leq i \leq I. 
\]
\end{theo}
We deduce from this theorem an efficient way to compute the long-time behavior of the scheme, which consists in seeking for $\left(({\phi_K^\infty)}_{K\in\Tt}, {(\mu_{i,\Tt}^\infty)}_{1\leq i \leq I} \right)$ belonging to $\operatorname{argmin} \Psi_\Tt$ and to a posteriori reconstruct the volume fraction profiles thanks to the functions $v_i$ defined by~\eqref{eq:vi}. 
We also remark that the uniqueness of the long-time limit is established in the discrete setting, while the question remains open in the continuous setting. Conjecturing the uniqueness of such an asymptotic limit for the continuous model, which then has to be the unique minimizer of $\Psi$ defined by~\eqref{eq:steady.Psi}, then Theorem~\ref{theo:main.3} illustrates in yet another aspect the structure-preserving nature of our scheme. 

The rest of the paper is organized as follows. Section~\ref{sec:num_an_fixed_grid} is devoted to the proof of Theorem~\ref{theo:main.1}, as well as to some post-treatment of the energy / energy dissipation inequality \eqref{eq:disc.EDI} in order to derive some uniform bounds on quantities to be used in the convergence analysis carried out in Section~\ref{sec:conv}, where the proof of Theorem~\ref{theo:main.2} is detailed. Our last main theoretical result Theorem~\ref{theo:main.3} is proved in Section~\ref{sec:asymptotic}.
The numerical results presented in Section~\ref{sec:num} contain some numerical evidence of the convergence of the scheme. They also illustrate that the situation is more complex than what Theorem~\ref{theo:main.3} suggests at first glance. Although the approximate volume fractions are (strictly) positive as pointed out in Theorem~\ref{theo:main.1}, $u_{0,\Tt,\tau}$ can take very small values in some parts of the domain. 
The discrete fluxes~\eqref{eq:scheme.Fi} then take very small values as well, making the evolution extremely slow. 
As the linear and non-linear systems corresponding to the scheme are merely solved approximately, one furthermore has to pay a particular attention to preserve the positivity at the discrete level.


\section{Uniform a priori bounds and existence of a discrete solution
    } \label{sec:num_an_fixed_grid}

    The goal of this section is to prove Theorem~\ref{theo:main.1}, i.e. to show that the non-linear system corresponding to the scheme~\eqref{eq:scheme.Poisson}--\eqref{eq:u0Kn} admits at least one solution, and that beyond local conservativity, this solution preserves at the discrete level some key features of the model, namely the positivity of the volume fractions, as well as uniform $L^{\infty}$ bounds for them, and the decay of the free energy. Furthermore, in this part, precise quantification of the dissipated total free energy supplies enough uniform estimates, in the mesh size $h_\Tt$ and time steps $\left( \tau_n \right)_{n \geq 1}$, to perform the convergence analysis in Section \ref{sec:conv}. 
    The grid $\Tt$ and the time steps $\left(\tau^n\right)_{n\geq 1}$ remain fixed throughout this section.

\subsection{Uniform a priori bounds and existence of a discrete solution}
\label{ssec:exist}
The first Lemma regards the discrete counterpart of the conservation of the total mass of the system.
\begin{lem}[Mass conservation]\label{lem:mass}
    It holds that
    \be\label{eq:ui.mass}
\sum_{K\in\Tt} u_{i,K}^n m_K = \sum_{K\in\Tt} u_{i,K}^{n-1} m_K = \sum_{K\in\Tt} u_{i,K}^0 m_K = \int_\O u_i^0 >0, \qquad 0 \leq i \leq I. 
\ee
\end{lem}
\begin{proof}
    Since our scheme is locally conservative, i.e., $F_{i,K\sig}^n + F_{i,L\sig}^n= 0$ for all $\sig=K|L \in \Ee_{\rm int}$, $i=1,\dots,I$, then summing~\eqref{eq:fvscheme_conserv} over $K$ shows by induction and thanks to~\eqref{eq:uiK0} the result. 
\end{proof}
Since we are interested in discrete solutions with positive volume fractions $u_{i,K}^n$, we perform an eventually harmless modification of the flux formula~\eqref{eq:scheme.Fi} into 
\be\label{eq:scheme.Fi.3}
F_{i,K\sigma}^n = a_\sig D_i\left(\left(u^n_{i,K}\right)^+\! \left(u^{n}_{0,L}\right)^+\!\B\left(z_i(\phi_L^n - \phi_K^n)\right) - \left(u^n_{i,L}\right)^+\!\left(u^{n}_{0,K}\right)^+\!\B\left(z_i(\phi_K^n - \phi_L^n)\right)\right).
\ee
\begin{prop}[Positivity of volume fractions]\label{prop:positivity}
Let $n\geq 1$, and let $\left(U_K^{n-1}\right)_{K\in\Tt}$ be such that 
\be\label{eq:uKn-1}
u_{i,K}^{n-1} \geq 0,\quad \sum_{i=0}^I u_{i,K}^{n-1} = 1 \quad \forall \, K \in \Tt, \quad \text{and}\quad \sum_{K\in\Tt} u_{i,K}^{n-1} m_K >0. 
\ee
Then any solution $\left(U_K^n, \phi_K^n \right)_{K\in\Tt,n\geq 1}$ to the modified scheme with~\eqref{eq:scheme.Fi.3} instead of \eqref{eq:scheme.Fi} satisfies $u_{i,K}^n >0$ for all $i=0,\dots, I$ and all $K\in \Tt$.
\end{prop}
\begin{proof}
Let us start by establishing the positivity of $u_{0,K}^n$. Assume for contradiction that there exists a cell $K\in\Tt$ such that $u_{0,K}^n \leq 0$. Then we deduce from formula~\eqref{eq:scheme.Fi.3} that $F_{i,K\sig}^n \geq 0$ for all $\sig \in \Ee_K$ and all $i =1,\dots, I$. Because of~\eqref{eq:u0Kn} and~\eqref{eq:uKn-1}, this implies that 
\[
0 \geq u_{0,K}^n = u_{0,K}^{n-1} + \frac{\tau^n}{m_K} \sum_{i=1}^I \sum_{\sig \in \Ee_K} F_{i,K\sig}^n \geq 0. 
\]
In particular, all the fluxes $F_{i,K\sig}^n$, $i=1,\dots, I$ and $\sig \in \Ee_K$, are equal to $0$. In view of formula~\eqref{eq:scheme.Fi.3} and of the strict positivity of $\B$, this implies either that $u_{i,K}^n \leq 0$ for all $i$, which yields a contradiction with~\eqref{eq:u0Kn}, or that $u_{0,L}^n \leq 0$ for all the cells $L$ sharing an edge $\sig = K|L$ with $K$. Since $\O$ is connected, one would obtain that $u_{0,K}^n = 0$ for all $K \in \Tt$ and thus that $\sum_{K\in\Tt} u_{0,K}^n m_K = 0$. This contradicts \eqref{eq:ui.mass}, and thus we necessarily have that $u_{0,K}^n >0$ for all $K\in\Tt$. 

With the positivity of $u_{0,K}^n$, $K\in\Tt$, at hand, let us focus on the $u_{i,K}^n$ for an arbitrary $i=1,\dots, I$. Similarly, we assume that there exists some $K\in\Tt$ such that $u_{i,K}^n \leq 0$. Then owing to~\eqref{eq:scheme.Fi.3}, we infer that $F_{i,K\sig}^n \leq 0$ for all $\sig \in \Ee_K$, and then that 
\[
0 \geq u_{i,K}^n = u_{i,K}^{n-1} -  \frac{\tau^n}{m_K} \sum_{\sig \in \Ee_K} F_{i,K\sig}^n \geq 0.
\]
This leads to $u_{i,K}^n = 0$ and to $F_{i,K\sig}^n = 0$ for all $\sig \in \Ee_K$.
Since we already know that $u_{0,K}^n >0$, we deduce from \eqref{eq:scheme.Fi.3} that $u_{i,L}^n \leq 0$ for all cells $L$ sharing a face $\sigma = K|L$ with $K$. As above, this implies as $u_{0,K}^n = 0$ for all $K\in\Tt$, which contradicts~\eqref{eq:ui.mass}. Then $u_{i,K}^n >0$ for all $K \in \Tt$, concluding the proof of Proposition~\ref{prop:positivity}. 
\end{proof}

A consequence of the previous proposition is that a solution to the modified scheme with \eqref{eq:scheme.Fi.3} instead of \eqref{eq:scheme.Fi} is also a solution to the original scheme~\eqref{eq:scheme.Poisson}--\eqref{eq:u0Kn}. 

As we did assume that the {background charge density} $f$ and thus its discrete counterpart  $\left(f_K\right)_{K\in\Tt}$ are uniformly bounded, and that $\phi^D$ belongs to $L^\infty\cap H^{1/2}(\Gamma^D)$,
we deduce 
a uniform bound for the electric potential.
\begin{prop}[Uniform bound for the electric potential]\label{prop:bound.phi}
There exists $C=C(\phi^D, \O, \lambda, f, {(z_i)}_{i}, \zeta_\Tt)$ such that
\be\label{eq:phi.Linf}
\max_{K\in\Tt} |\phi_K^n| + \left(\sum_{\sigma \in \Ee} a_\sigma \left( \phi_K^n - \phi_{K\sig}^n\right)^2\right)^{1/2} \leq C, \qquad n \geq 1.
\ee
\end{prop}
The proof of this result can be found in~\cite[Proposition A.1]{CCFG21} and is a consequence of the uniform boundedness of the right-hand side of the discrete Poisson equation~\eqref{eq:scheme.Poisson}. 

These a priori estimates are sufficient to prove the existence of a solution to the scheme. 
We end up with the following proposition.

\begin{prop}[Existence of solutions]\label{prop:existence}
There exists at least one solution to the numerical scheme~\eqref{eq:scheme.Poisson}--\eqref{eq:u0Kn} such that 
$u_{i,K}^n >0$ for all $i = 0,\dots, I$, for all $K \in \Tt$ and all $n\geq 1$. 
\end{prop}

\begin{proof}
    The proof is based on an inductive procedure in time. At each time step $n \geq 1$, we use a topological degree argument to show the existence of a solution $(U^n_K,\phi^n_K)_{K \in \Tt} \in (0,1)^{I \times \Tt} \times [-C,C]^\Tt$ to \eqref{eq:scheme.Poisson}--\eqref{eq:fvscheme_conserv}, with the fluxes defined as in \eqref{eq:scheme.Fi.3}. Because of the positivity of the discrete volume fractions stated in Proposition~\ref{prop:positivity} and to the equivalence between 
    \eqref{eq:scheme.Fi.3} and \eqref{eq:scheme.Fi} for such positive volume fractions, this ensures the existence for the original scheme \eqref{eq:scheme.Poisson}--\eqref{eq:scheme.Fi}.

    Let $n \geq 1$ be such that $(U^{n-1}_K,\phi^{n-1}_K)_{K \in \Tt} \in [0,1]^{I \times \Tt} \times [-C,C]^\Tt$ is given (that is the case for $n=1$, thanks to \eqref{eq:uiK0} and \eqref{eq:init}).
    The idea is to deform our non-linear system continuously until it is transformed into a linear one with known solutions. Therefore, let introduce a parameter $s \in [0, \tau^n]$ and define, for every $K\in\Tt$,
    \begin{equation}
    \label{eq:homotoy_def}
        \begin{cases}
    m_K \left(u_{i,K}^{(s)} - u_{i,K}^{n-1}\right) + s \sum_{\sigma\in\Ee_K} F_{i,K\sigma}^{(s)} = 0, \quad i = 1,\dots, I, \\[6pt]
    \lambda^2 \sum_{\sig \in \Ee_K} a_\sig \left( \phi_K^{(s)} - \phi_{K\sig}^{(s)}\right) = m_K\left( f_K + \sum_{i=1}^I z_i u_{i,K}^{(s)} \right),  \end{cases}
    \end{equation}        
    with $F_{i,K\sigma}^{(s)}$ defined by \eqref{eq:scheme.Fi.3}, which corresponds to the original schemes when $s= \tau^n$.
    For $s=0$, the two equations can be decoupled and the first one can be rewritten in matrix form
    \[\mathbb{M} (U^{(0)} - U^{n-1})=0,\]
    where $U^{(s)}=\left(U^{(s)}_K\right)_{K \in \Tt}$, for $s \in [0,\tau^n]$, $U^{n-1}=\left(U^{n-1}_K\right)_{K \in \Tt}$ and $\mathbb{M}=\text{diag}\left((m_K)_{K \in \Tt}\right)$ is a positive definite matrix. Therefore, there exists a unique solution $U^{(0)} = U^{n-1} \in [0,1]^{I \times \Tt}$. Via the Proposition~\ref{prop:bound.phi}, there also exists a unique $\phi^{(s)}=\left( \phi^{(s)}_K \right)_{K \in \Tt} \in [-C, C]^{\Tt}$. Moreover, thanks to the continuity of the discrete fluxes, the functional
    \begin{equation*}
         \mathcal{S} \; : \; \begin{cases}
            [0,\tau^n] \times [-2,2]^{I \times \Tt} \times [-C-1,C+1]^\Tt \to \R^\Tt \times \R^\Tt \\
            (s, U, \phi) \mapsto  \mathcal{S}(s, U, \phi)
        \end{cases}
    \end{equation*}
    whose zeros $(s, U^{(s)}, \phi^{(s)})$ are the solutions of \eqref{eq:homotoy_def}, is hence an homotopy. Furthermore, all along it, its zeros $(U^{(s)}, \phi^{(s)})$ remain inside the compact subset $[0,1]^{I \times \Tt} \times [-C,C]^\Tt$. 

    Thus, the topological degree corresponding to $\mathcal{S}(s, U, \phi)=0$ is equal to one, all along the homotopy, and hence in particular for $s=\tau^n$. That implies the existence of a solution to the scheme (but it does not guarantee uniqueness).
\end{proof}
We note that the proof proposed here is simpler than that found in, for example, ~\cite[Proposition 3.2]{CCFG21}. The reason lies in the fact that we did not use an entropy inequality to derive the uniform estimates on volume fractions and, therefore, we do not have to guarantee the validity of it during the homotopy.

\begin{rema}[On the uniqueness of solutions] 
The uniqueness of weak solutions for \eqref{eq:contmodel_conserv}--\eqref{eq:init}, as well as for numerical schemes found in the literature, remains largely unresolved. This challenge is inherent to the cross-diffusion structure of the problem. In \cite{GJ18}, uniqueness is established under strong assumptions on the diffusion coefficients and valences, specifically $D_i \equiv D$ and $z_i \equiv z$ for all $i = 1, \dots, I$.

For the numerical scheme \eqref{eq:fvscheme_conserv}--\eqref{eq:uiK0}, by following the argument in \cite[Section 3.2, Uniqueness of solutions]{CCGJ19}, uniqueness of discrete solutions can also be shown under the same assumptions: constant diffusion coefficients $D_i \equiv D$ and the absence of a drift term in the fluxes. The assumption of constant diffusion coefficients enables a preliminary proof of uniqueness for the discrete variable $u_0$. In our scheme, since no drift term is present and $\B(0) = 1$, the function $\B$ does not appear in the definition of the fluxes. Consequently, the equation for the discrete $u_0$ is significantly simpler than that in \cite{CCGJ19}, and uniqueness of $u_0$ follows straightforwardly. The absence of a drift term also allows the application of Gajewski’s entropy method to deduce uniqueness for the remaining components $(u_i)_{1 \leq i \leq I}$, as in \cite{CCGJ19}.

\end{rema}

\subsection{Energy dissipation at the discrete level}\label{ssec:NRG}
The next proposition is about the thermodynamical consistency of our scheme and the decay of a discrete counterpart of the free energy. 
\begin{prop}[Free energy decay]\label{prop:NRG}
Let $\left(U_K^n, \phi_K^n \right)_{K\in\Tt,n\geq 1}$ be a solution to the scheme \eqref{eq:scheme.Poisson}--\eqref{eq:u0Kn} as in Proposition~\ref{prop:existence}, then define for $n\geq 0$ the discrete free energy at the $n^\text{th}$ time step 
\begin{equation}\label{eq:free_energ_discr}
\Hh_\Tt^n = \sum_{K\in\Tt} m_K H(U_K^n) + \frac{\lambda^2}2 \sum_{\sig \in \Ee} a_\sig (\phi_K^n - \phi_{K\sig}^n)^2 - \lambda^2 \sum_{\sig \in \Ee^D} a_\sig \phi_\sig^D 
(\phi_\sig^D - \phi_K^n), 
\end{equation}
the discrete electrochemical potentials  $\mu_{i,K}^n = \log\left(\frac{u_{i,K}^n}{u_{0,K}^n}\right) + z_i \phi_K^n$ of species $i$, 
and 
\[
\Dd_\Tt^n = \sum_{i=1}^I \sum_{\sig \in \Ee_{\rm int}}F_{i,K\sig}^n (\mu_{i,K}^n - \mu_{i,L}^n)
\]
the discrete dissipation, which vanishes if and only if ${(\mu_{i,K}^n)}_{K\in\Tt}$ is constant for all $i=1, \dots, I$. 
 Then there holds
\be\label{eq:NRG.1}
\Hh_\Tt^n + \tau^n \Dd_\Tt^n  \leq \Hh_\Tt^{n-1}, \qquad n \geq 1.
\ee
\end{prop}
\begin{proof}
As a consequence of the positivity of $u_{0,K}^n$ and of the monotonicity of the exponential function, one then easily infers from reformulation~\eqref{eq:scheme.Fi.2} of the discrete fluxes that 
\be\label{eq:Disign}
\Dd_{i,\sig}^n :=F_{i,K\sigma}^n(\mu_{i,K}^n - \mu_{i,L}^n) \geq 0, \qquad \forall i= 1,\dots, I, \; \sig = K|L \in \Ee_{{\rm int}},
\ee
whence the non-negativity of $\Dd^n$. As $u_{0,K}^n >0$ for all $K\in\Tt$ and $n\geq 1$, and as $y\mapsto e^y$ is strictly increasing, one gets that $\Dd_{i,\sig}^n = 0$ iff $\mu_{i,K}^n = \mu_{i,L}^n$.

To proceed, we first define by $\mu_{i,K}^n = \log\left(\frac{u_{i,K}^n}{u_{0,K}^n}\right) + z_i \phi_K^n = \log(w_{i,K}^n)$ the electrochemical potential of species $i$, then we multiply
the discrete conservation law~\eqref{eq:fvscheme_conserv} by $\tau^n \mu_{i,K}^n$ and sum over $i=1,\dots, I$ and $K\in\Tt$. Thanks to the local conservativity of the fluxes, i.e., $F_{i,K\sig}^n + F_{i,L\sig}^n = 0$ for all $\sig = K|L \in \Ee_{\rm int}$, and to the no-flux boundary condition $F_{i,K\sig}^n = 0$ for $\sigma\in \Ee_\text{ext}$, $i = 1, \dots, I$, a discrete integration by parts in space allows the flux terms to be rewritten as a sum over interior edges only, namely $ \sum_{K\in\Tt} \sum_{\sigma\in\Ee_K} F_{i,K\sigma}^n \mu_{i,K}^n = \sum_{\sig \in \Ee_{\rm int}} \Dd_{i,\sig}^n$ for all $i = 1, \dots, I$. This leads to the identity
\be\label{eq:ABDn}
\Aa_\Tt^n + \Bb_\Tt^n + \tau^n \Dd_\Tt^n = 0.
\ee
where we have set
\begin{align*}
\Aa_\Tt^n =& \;\sum_{i=1}^I \sum_{K\in\Tt} \left(u_{i,K}^n - u_{i,K}^{n-1}\right) \log\left(\frac{u_{i,K}^n}{u_{0,K}^n}\right)m_K \\
\overset{\eqref{eq:u0Kn}}= &  \sum_{i=0}^I \sum_{K\in\Tt} \left(u_{i,K}^n - u_{i,K}^{n-1}\right) \frac{\partial \widetilde H}{\partial u_i} (u_{0,K}^n \dots, u_{I,K}^n) m_K ,
\end{align*}
where $\widetilde H : \left\{ (u_i)_{0\leq i\leq I}  \; \middle| \; \sum_{i=0}^I u_i =1\right\} \to [-\log(I+1),+\infty)$, defined by 
\[
\widetilde H(u_0,\dots,u_I):= \sum_{i=0}^I u_i\log(u_i),\] and
\begin{align*}
\Bb_\Tt^n = & \; \sum_{i=1}^I \sum_{K\in\Tt} \left(u_{i,K}^n - u_{i,K}^{n-1}\right) z_i \phi_K^n m_K \\\overset{\eqref{eq:scheme.Poisson}}=&
\lambda^2 \sum_{K\in\Tt} \phi_K^n \sum_{\sig \in \Ee_K} a_\sig \left(\phi_K^n - \phi_K^{n-1} - (\phi_{K\sig}^n - \phi_{K\sig}^{n-1})\right).
\end{align*}
From the convexity of $\widetilde H$ and the identity $\widetilde H(u_0,\dots,u_I)= H(U)$, we deduce that
\be\label{eq:An}
\Aa_\Tt^n\geq \sum_{K\in\Tt}\left( H(U_K^n) - H(U_K^{n-1}) \right) m_K,
\ee
while the term $\Bb^n_\Tt$ can be rewritten by observing that each interior edge is counted twice in the sum $\sum_{K \in \Tt} \sum_{\sig \in \Ee_K}$. The resulting equivalent reformulation is
\begin{align*}
\Bb_\Tt^n =&\; \lambda^2  \sum_{\sig \in \Ee} a_\sig \left(\phi_K^n - \phi_K^{n-1} - (\phi_{K\sig}^n - \phi_{K\sig}^{n-1})\right)(\phi_K^n - \phi_{K\sig}^n)
\\&\phantom{xxxxx} \;+ \lambda^2 \sum_{\sig \in \Ee^D} a_\sig \phi_\sig^D 
(\phi_K^n - \phi_{K}^{n-1}),
\end{align*}
which is valid since boundary edges $\sig \in \Ee^N$ do not contribute to the first sum, and the contributions from boundary edges $\sig \in \Ee^D$ cancel with the second part of the new expression of $\Bb_\Tt^n$.
Then using the elementary convexity inequality $a(a-b) \geq (a^2 - b^2)/2$ in the above term and combining the result with~\eqref{eq:An} in \eqref{eq:ABDn} provides the desired result~\eqref{eq:NRG.1}.
\end{proof}

Proposition~\ref{prop:NRG} allows us to complete the proof of Theorem~\ref{theo:main.1}, but it also contains important information for proving the convergence of the scheme. It is the purpose of the next section to extract this information.

\subsection{Further uniform estimates on the discrete solution}\label{ssec:estimates++} 
To pass to the limit in the scheme and to identify the limit as a weak solution, we need to extract some further uniform estimates, in particular the discrete $L^2_{\rm loc}(H^1)$ estimates on the discrete counterparts of $u_0$ and $\sqrt{u_i u_0}$. {We prove these estimates in Lemma~\ref{lem:L2locH1}. As an intermediate result we need a uniform bound on the discrete free energy. }

\begin{lem}[Uniform bound for the free energy]\label{lem:boundentropy}
{There exists $C>0$ depending only on $\O$, $\phi^D$, $\lambda$, $f$, ${(z_i)}_i$, and $\zeta_\Tt$ such that, for all $N\geq 1$, there holds 
$
|\Hh_\Tt^N|\leq C
$.}
\end{lem}
\begin{proof}
Because of the bound $0\leq u_{i,K}^n \leq 1$ for all $i$ and $K$, it is clear that the first two contributions of~\eqref{eq:free_energ_discr} remain uniformly bounded by a constant depending only on $\O$. Indeed, for the first term, this is a direct consequence of the definition~\eqref{eq:H} of $H$, whereas the second term can be controlled thanks to Proposition~\ref{prop:bound.phi}.  Concerning the last contribution observe that if one defines $\phi_K^D$ and $\phi_\sigma^D$ as the averages of $\phi^D$ on $K\in\Tt$ and $\sigma\in\Ee$ respectively, then a suitable reorganization of the sum shows that 
\begin{multline}\label{eq:boundentropy.1}
  \sum_{\sig \in \Ee^D} a_\sig \phi_\sig^D 
(\phi_K^n - \phi_\sig^D) =   \sum_{\sig \in \Ee} a_\sig (\phi_\sig^D - \phi_K^D)
(\phi_K^n - \phi_{K\sig}^n) + \sum_{K\in\Tt}\phi_K^D\sum_{\sig \in \Ee_K} a_\sig 
(\phi_K^n - \phi_{K\sig}^n) \\ 
\overset{\eqref{eq:scheme.Poisson}}{=} 
 \sum_{\sig \in \Ee} a_\sig (\phi_\sig^D - \phi_K^D)
(\phi_K^n - \phi_{K\sig}^n) + \frac1{\lambda^2}  \sum_{K\in\Tt} m_K \phi_K^D \left(f_K + \sum_{i=1}^I z_i u_{i,K}^n\right).
\end{multline}
Using Young's inequality and $\Ee^D \subset \Ee$ for the first term in the above right-hand side gives 
\[
 \sum_{K\in\Tt}\sum_{\sig \in \Ee^D} a_\sig (\phi_\sig^D - \phi_K^D)
(\phi_K^n - \phi_{K\sig}^n) \leq  \sum_{\sig \in \Ee} a_\sig (\phi_\sig^D - \phi_K^D)^2 
+ \frac14 \sum_{\sig \in \Ee} a_\sig (\phi_K^n - \phi_{K\sig}^n)^2.
\]
Then Lemma 9.4 of~\cite{Eymard2000} shows that 
\[
 \sum_{K\in\Tt}\sum_{\sig \in \Ee^D} a_\sig (\phi_\sig^D - \phi_K^D)^2 \leq C(\zeta_\Tt) \|\grad \phi^D\|_{L^2(\O)}^2.
\]
Then since $f$ and $U$ are bounded, we deduce from~\eqref{eq:boundentropy.1} that there exists $C$ depending only on 
$\|f\|_\infty$, $\max_i |z_i|$, $\O$, $\| \phi^D\|_{H^1}$, $\lambda$, and $\zeta_\Tt$ such that 
\[
  \sum_{\sig \in \Ee^D} a_\sig \phi_\sig^D (\phi_K^n - \phi_\sig^D) \leq C. 
\]
The result of the lemma follows. 
\end{proof}

The next lemma shows that the control of the energy dissipation $\sum_{n\geq 1} \tau^n \Dd_\Tt^n$ inferred from Proposition~\ref{prop:NRG} gives some $L^2_{\rm loc}(H^1)$ type control on the discrete counterparts of $u_0$, $\sqrt{u_0}$, and $\sqrt{u_iu_0}$. 

\begin{lem}[Uniform bounds on $u_0$, $\sqrt{u_0}$, and $\sqrt{u_iu_0}$]\label{lem:L2locH1}
There exists $C>0$ depending only on {$\O$}, $\phi^D$, $\lambda$, $f$, ${(z_i)}_i$,  ${(D_i)}_i$, and $\zeta_\Tt$ such that, for all $N\geq 1$, there holds
\begin{multline}\label{eq:L2locH1.1}
\sum_{n=1}^N \tau^n \sum_{i=1}^{I} \sum_{\sig \in \Ee_{\rm int}} a_\sig \left(\sqrt{u_{i,K}^n u_{0,K}^n}-\sqrt{u_{i,L}^n u_{0,L}^n}\right)^2 
\\+ \sum_{n=1}^N \tau^n \sum_{\sig \in \Ee_{\rm int}} a_\sig \left(\sqrt{u_{0,K}^n} - \sqrt{u_{0,L}^n}\right)^2 
\\
+ \sum_{n=1}^N \tau^n \sum_{\sig \in \Ee_{\rm int}} a_\sig \left(u_{0,K}^n - u_{0,L}^n\right)^2 
\leq C(1+t^N).\end{multline}
As a consequence, one also has 
\be\label{eq:L2locH1.2}
\sum_{n=1}^N \tau^n \sum_{i=1}^{I} \sum_{\sig \in \Ee_{\rm int}} a_\sig \left(u_{i,K}^n \sqrt{u_{0,K}^n}-u_{i,L}^n\sqrt{u_{0,L}^n}\right)^2 \leq C(1+t^N)
\ee
and 
\be\label{eq:L2locH1.4}
\sum_{n=1}^N \tau^n \sum_{i=1}^{I} \sum_{\sig \in \Ee_{\rm int}} a_\sig \left(u_{i,K}^n u_{0,K}^n-u_{i,L}^nu_{0,L}^n\right)^2 \leq C(1+t^N).
\ee
\end{lem}
\begin{proof}
One gets from the elementary inequality $(a-b)(\log(a) - \log(b)) \geq 4 (\sqrt a - \sqrt b )^2$ applied to~\eqref{eq:Disign}  that 
\[
\Dd_{i,\sig}^n \geq 4 a_\sig D_i \Rf(e^{-z_i \phi_K^n}, e^{-z_i \phi_L^n})  \left( \sqrt{u_{i,K}^n u_{0,L}^n} e^{\frac{z_i}4 (\phi_K^n-\phi_L^n)} 
- \sqrt{u_{i,L}^n u_{0,K}^n} e^{\frac{z_i}4(\phi_L^n-\phi_K^n)} \right)^2
\]
with $\Rf(e^{-z_i \phi_K^n}, e^{-z_i \phi_L^n})  =\M(e^{-z_i \phi_K^n}, e^{-z_i \phi_L^n}) e^{\frac{z_i}2(\phi_K^n + \phi_L^n)}$.
Thanks to~\eqref{eq:phi.Linf} and since $D_i>0$ for all $i$, there holds 
\[
2 D_i \Rf(e^{-z_i \phi_K^n}, e^{-z_i \phi_L^n}) \geq \kappa 
\]
for some  $\kappa>0$ uniform w.r.t. $K$, $i$, and $n$. As a consequence, using furthermore that  $(a+b)^2 \geq \frac12 a^2 - b^2$, 
\begin{multline*}
\Dd_{i,\sig}^n \geq  \kappa a_\sig \cosh^2\left(\frac{z_i}4 (\phi_K^n-\phi_L^n)\right) \left(\sqrt{u_{i,K}^n u_{0,L}^n}-\sqrt{u_{i,L}^n u_{0,K}^n}\right)^2\\
-  2\kappa a_\sig  \left(\sqrt{u_{i,K}^n u_{0,L}^n}+\sqrt{u_{i,L}^n u_{0,K}^n} \right)^2 \sinh^2\left(\frac{z_i}4 (\phi_K^n-\phi_L^n)\right).
\end{multline*}
Since $|\phi_{K}^n|\leq C$ owing to \eqref{eq:phi.Linf}, one has $\sinh^2\left(\frac{z_i}4 (\phi_K^n-\phi_L^n)\right)\leq C (\phi_K^n-\phi_L^n)^2$. Using moreover that $0 < u_{i,K}^n, u_{0,K}^n < 1$ and that $\cosh(a) \geq 1$, one gets that 
\[
\Dd_{i,\sig}^n \geq  a_\sig \kappa  \left(\sqrt{u_{i,K}^n u_{0,L}^n}-\sqrt{u_{i,L}^n u_{0,K}^n}\right)^2 - C a_\sig  (\phi_K^n-\phi_L^n)^2.
\]
Since 
\be\label{eq:L2locH1.3}
 \left(\sqrt{u_{i,K}^n u_{0,L}^n}-\sqrt{u_{i,L}^n u_{0,K}^n}\right)^2 =  \left(\sqrt{u_{i,K}^n u_{0,K}^n}-\sqrt{u_{i,L}^n u_{0,L}^n}\right)^2 - (u_{i,K}^n - u_{i,L}^n)(u_{0,K}^n - u_{0,L}^n),
\ee
then summing over $i = 1,\dots, I$ and $\sig = K|L$ and using \eqref{eq:u0Kn} leads to 
\begin{multline*}
\Dd_\Tt^n \geq \kappa \sum_{i=1}^{I} \sum_{\sig \in \Ee_{\rm int}} a_\sig \left(\sqrt{u_{i,K}^n u_{0,K}^n}-\sqrt{u_{i,L}^n u_{0,L}^n}\right)^2 
\\
+ \kappa \sum_{\sig \in \Ee_{\rm int}} a_\sig \left(u_{0,K}^n - u_{0,L}^n\right)^2 
-C \sum_{\sig \in \Ee_{\rm int}} a_\sig (\phi_K^n - \phi_{K\sig}^n)^2.
\end{multline*}
Bearing in mind Proposition~\ref{prop:bound.phi}, we obtain that 
\be\label{eq:Dn.1}
\Dd_\Tt^n \geq \kappa \sum_{i=1}^{I} \sum_{\sig \in \Ee_{\rm int}} a_\sig \left(\sqrt{u_{i,K}^n u_{0,K}^n}-\sqrt{u_{i,L}^n u_{0,L}^n}\right)^2 
+ \kappa \sum_{\sig \in \Ee_{\rm int}} a_\sig \left(u_{0,K}^n - u_{0,L}^n\right)^2 
-C.
\ee
Moreover, 
since $\sqrt{ab} \leq (a+b)/2$ for all $a,b \geq 0$, there holds
\[
\sum^I_{i=0} \sqrt{u_{i,K}^n u_{i,L}^n} \leq \sum^I_{i=0} \frac{u_{i,K}^n +u_{i,L}^n}2 = 1, \qquad \forall \sigma = K|L \in \Ee_{\text{int}}. 
\] 
This provides
\begin{multline*}
 \sum_{i=1}^{I} \sum_{\sig \in \Ee_{\rm int}} a_\sig \left(\sqrt{u_{i,K}^n u_{0,K}^n}-\sqrt{u_{i,L}^n u_{0,L}^n}\right)^2 \\
 \geq   \sum_{\sig \in \Ee_{\rm int}} a_\sig \left( (1-u_{0,K}^n) u_{0,K}^n + (1-u_{0,L}^n) u_{0,L}^n - 2 (1-\sqrt{u_{0,K}^n u_{0,L}^n})\sqrt{u_{0,K}^n u_{0,L}^n}
 \right)
 \\
= \sum_{\sig \in \Ee_{\rm int}} a_\sig \left(\sqrt{u_{0,K}^n}-\sqrt{u_{0,L}^n}\right)^2 -  \sum_{\sig \in \Ee_{\rm int}} a_\sig \left({u_{0,K}^n}-{u_{0,L}^n}\right)^2, 
\end{multline*}
whence we also deduce that 
\[
\Dd_\Tt^n \geq \kappa  \sum_{\sig \in \Ee_{\rm int}} a_\sig \left(\sqrt{u_{0,K}^n}-\sqrt{u_{0,L}^n}\right)^2 - C. 
\]
{To deduce \eqref{eq:L2locH1.1}, it eventually remains to remark from \eqref{eq:NRG.1} and Lemma~\ref{lem:boundentropy} that there exists $C$ depending on none of $h$, $\Delta t$, $N$ nor on the initial data $U^0 = \left(u_i^0\right)_{1 \leq i \leq I}$ (provided it fulfills~\eqref{eq:init}) such that $\sum_{n=1}^N \tau^n \Dd_\Tt^n \leq C.$ Combining this with~\eqref{eq:Dn.1} yields~\eqref{eq:L2locH1.1}.}

To recover estimate~\eqref{eq:L2locH1.2}, remark that     \begin{align*}
        \Big(\sqrt{u^n_{0,K}}u^n_{i,K} -\sqrt{u^n_{0,L}}u^n_{i,L}\Big)  = \Bigg[&  \sqrt{u^n_{i,K}} \Big( \sqrt{u^n_{0,K} u^n_{i,K}} - \sqrt{u^n_{0,L} u^n_{i,L}}\Big) \\
        & \quad + u^n_{i,L} \Big( \sqrt{u^n_{0,K}} - \sqrt{u^n_{0,L}}\Big) \\
        & \quad +\sqrt{u^n_{i,L}} \Big( \sqrt{u^n_{0,L} u^n_{i,K}} - \sqrt{u^n_{0,K} u^n_{i,L}}\Big) \Bigg].
    \end{align*}
As $(a+b+c)^2 \leq 3(a^2 + b^2 + c^2)$, and since $0 \leq u_{i,K}^n \leq 1$, one gets that 
\begin{align*}
\frac13 \Big(\sqrt{u^n_{0,K}}u^n_{i,K} -\sqrt{u^n_{0,L}}u^n_{i,L}\Big)^2 \leq &\Big( \sqrt{u^n_{0,K} u^n_{i,K}} - \sqrt{u^n_{0,L} u^n_{i,L}}\Big)^2 \\ & + \Big( \sqrt{u^n_{0,K}} - \sqrt{u^n_{0,L}}\Big)^2 \\
& +  \Big( \sqrt{u^n_{0,L} u^n_{i,K}} - \sqrt{u^n_{0,K} u^n_{i,L}}\Big)^2.
\end{align*}
The two first terms in the right-hand side are controlled thanks to~\eqref{eq:L2locH1.1}, while, bearing in mind~\eqref{eq:L2locH1.3}, the third term can be overestimated as follows:
\begin{align*}
 \Big( \sqrt{u^n_{0,L} u^n_{i,K}} - \sqrt{u^n_{0,K} u^n_{i,L}}\Big)^2 \leq \; &\sum_{i=1}^I  \Big( \sqrt{u^n_{0,L} u^n_{i,K}} - \sqrt{u^n_{0,K} u^n_{i,L}}\Big)^2 \\
\leq\; & \sum_{i=1}^I  \Big( \sqrt{u^n_{0,K} u^n_{i,K}} - \sqrt{u^n_{0,L} u^n_{i,L}}\Big)^2 + (u_{0,K}^n - u_{0,L}^n)^2, 
\end{align*}
which again, after summation on $n=1,\dots, N$ and $\sig \in \Ee_{\rm int}$, can be controlled thanks to~\eqref{eq:L2locH1.1}.

Finally, to establish~\eqref{eq:L2locH1.4}, remark that 
\[
u_{i,K}^nu_{0,K}^n - u_{i,L}^nu_{0,L}^n = \left(\sqrt{u_{i,K}^nu_{0,K}^n} - \sqrt{u_{i,L}^nu_{0,L}^n}\right)
\left(\sqrt{u_{i,K}^nu_{0,K}^n} + \sqrt{u_{i,L}^nu_{0,L}^n}\right). 
\]
Since $0 \leq u_{0,K}^n , u_{i,K}^n \leq 1$, one gets that 
\[
\left(u_{i,K}^n u_{0,K}^n-u_{i,L}^nu_{0,L}^n\right)^2 \leq 4  \left(\sqrt{u_{i,K}^nu_{0,K}^n} - \sqrt{u_{i,L}^nu_{0,L}^n}\right)^2.
\]
Summing over $n=1,\dots, N$ and $\sig \in \Ee_{\rm int}$ and using~\eqref{eq:L2locH1.1} gives the desired result. 
\end{proof}

One also deduces the following discrete $L^2_{\rm loc}(L^2)^d$ estimates on the fluxes, which amount to some discrete $L^2_{\rm loc}(H^1)'$ estimate on time increments of the discrete counterpart to $\partial_t u_i$.
\begin{lem}[Uniform bounds on the fluxes]\label{lem:FiL2}
There exists $C$ depending only on $\O$, $\phi^D$, $\lambda$, $f$, ${(z_i)}_i$,  ${(D_i)}_i$, and $\zeta_\Tt$ such that 
\be\label{eq:FiL2}
\sum_{i = 1}^I \sum_{n=1}^N \tau^n \sum_{\sig \in \Ee_{\rm int}} \frac{d_\sig}{m_\sig} \left|F_{i,K\sig}^n \right|^2 
\leq C(1+t^N). 
\ee
\end{lem}
\begin{proof}
To obtain discrete $L^2_{\rm loc}(L^2)^d$ estimates on the fluxes, we need to exploit the discrete uniform estimates we have for $(U_K^n,\phi_K^n)_{K \in \Tt, n \geq 1}.$ In this regard, we rewrite the fluxes in a different way before taking advantage of some properties of the Bernoulli function~\eqref{eq:B.SG}.
One splits the flux~\eqref{eq:scheme.Fi} into two parts corresponding to convection and diffusion respectively:
\[
F_{i,K\sig}^n = F_{i,K\sig}^{\text{conv},n} + F_{i,K\sig}^{\text{diff},n},
\]
with 
\begin{align*}
 F_{i,K\sig}^{\text{conv},n} =\;& a_\sig D_i \frac{u_{i,K}^n u_{0,L}^n + u_{i,L}^n u_{0,K}^n}2 \left[ \B\big(z_i(\phi_L^n- \phi_K^n)\big) -  \B\big(z_i(\phi_K^n- \phi_L^n)\big) \right],\\
  F_{i,K\sig}^{\text{diff},n} =\;& a_\sig D_i \frac{u_{i,K}^n u_{0,L}^n - u_{i,L}^n u_{0,K}^n}2 \left[ \B\big(z_i(\phi_L^n- \phi_K^n)\big) +  \B\big(z_i(\phi_K^n- \phi_L^n)\big) \right].
\end{align*}
The flux $(F_{i,K\sig}^n)_{\sig, n}$ is bounded in $L^2_{\rm loc}(L^2)^d$ in the sense of~\eqref{eq:FiL2} if both  $( F_{i,K\sig}^{\text{conv},n})_{\sig,n}$ and $( F_{i,K\sig}^{\text{diff},n})_{\sig,n}$ are. 

As $\B(-y) - \B(y) = y$, one gets that 
\be\label{eq:Fi.conv}
 F_{i,K\sig}^{\text{conv},n} = a_\sig D_i \frac{u_{i,K}^n u_{0,L}^n + u_{i,L}^n u_{0,K}^n}2 z_i(\phi_K^n- \phi_L^n).
\ee
The $L^2_{\rm loc}(L^2)^d$ character of the above expression directly follows from the uniform bound on $u_{i,K}^n$, $0 \leq i \leq I$ and from the discrete $L^\infty(H^1)$ bound on $(\phi_K^n)_{K,n}$ stated in Proposition~\ref{prop:bound.phi}.
Therefore, 
\[
\sum_{i = 1}^I \sum_{n=1}^N \tau^n \sum_{\sig \in \Ee_{\rm int}} \frac{d_\sig}{m_\sig} \left|F_{i,K\sig}^{\text{conv},n} \right|^2 
\leq C t^N. 
\]

Concerning the diffusive term, as $2 \leq \B(y) + \B(-y) = y \coth \frac{y}2 \leq 2 +\frac{y^2}6 $, one has  that 
\be\label{eq:Fi.diff}
  F_{i,K\sig}^{\text{diff},n} = a_\sig D_i  \left(u_{i,K}^n u_{0,L}^n - u_{i,L}^n u_{0,K}^n \right)
  \left(1 + \Oo \left( \left(\phi_K^n - \phi_L^n\right)^2 \right)\right).
\ee
From the discrete $L^\infty(H^1)$ bound on $(\phi_K^n)_{K,n}$, one can uniformly estimate the remainder. 
For the other term,
since 
$u_{i,K}^n u_{0,L}^n - u_{i,L}^n u_{0,K}^n = u_{i,K}^n  u_{0,K}^n - u_{i,L}^n u_{0,L}^n 
+ (u_{i,K}^n+u_{i,L}^n) (u_{0,L}^n - u_{0,K}^n),$ 
then 
\begin{eqnarray*}
    |u_{i,K}^n u_{0,L}^n - u_{i,L}^n u_{0,K}^n|^2 &\leq& C \left( |u_{i,K}^n  u_{0,K}^n - u_{i,L}^n u_{0,L}^n|^2 + |(u_{i,K}^n+u_{i,L}^n) (u_{0,L}^n - u_{0,K}^n)|^2 \right) \\
&\leq & C \left( \left|\sqrt{u_{i,K}^n  u_{0,K}^n} - \sqrt{u_{i,L}^n u_{0,L}^n}\right|^2 
+ |u_{0,L}^n - u_{0,K}^n|^2 \right), 
\end{eqnarray*}
so that we obtain
\begin{eqnarray*}
\sum_{i = 1}^I \sum_{n=1}^N \tau^n \sum_{\sig \in \Ee_{\rm int}} \frac{d_\sig}{m_\sig} \left|F_{i,K\sig}^{\text{diff},n} \right|^2 
&\leq& C \sum_{i = 1}^I \sum_{n=1}^N \tau^n \sum_{\sig \in \Ee_{\rm int}} a_\sig \left|\sqrt{u_{i,K}^n  u_{0,K}^n} - \sqrt{u_{i,L}^n u_{0,L}^n}\right|^2 \\
&& + C \sum_{i = 1}^I \sum_{n=1}^N \tau^n \sum_{\sig \in \Ee_{\rm int}} a_\sig |u_{0,L}^n - u_{0,K}^n|^2 ,
\end{eqnarray*}
thanks to the uniform bounds on $(U^n_K)_{K \in \Tt, n \leq 1}$.
Lemma~\ref{lem:L2locH1} provides the desired $L^2_{\rm loc}(L^2)$ bound on $  F_{i,K\sig}^{\text{diff},n}$, concluding the 
proof of  Lemma~\ref{lem:FiL2}.
\end{proof}

From Lemma~\ref{lem:FiL2}, we deduce the following discrete $L^2_{\rm loc}((H^1)')$ estimate: 
\begin{cor}\label{cor:L2H-1}
There exists $C$ depending only on $\O$, $\phi^D$, $\lambda$, $f$, ${(z_i)}_i$,  ${(D_i)}_i$, and $\zeta_\Tt$ such that, 
for all $i = 0, \dots, I$ and all $\varphi_{\Tt,\tau} = \sum_{K\in\Tt} \sum_{n=1}^N \varphi_K^n \mathbbm{1}_{(t^{n-1},t^n] \times K}$, one has 
\be\label{eq:L2H-1}
\sum_{n=1}^N \sum_{K\in\Tt} m_K(u_{i,K}^n - u_{i,K}^{n-1}) \varphi_K^n \leq C\left(1+t^N\right)^{1/2} \left( \sum_{n=1}^N \tau^n \sum_{\sig \in \Ee_{\rm int}} a_\sig \left(\varphi_K^n - \varphi_{K\sig}^n\right)^2 \right)^{1/2}.
\ee
\end{cor}
\begin{proof}
Let us first establish~\eqref{eq:L2H-1} for $i \geq 1$. Multiplying~\eqref{eq:fvscheme_conserv} by $\tau^n \varphi_K^n$, summing over $K\in\Tt$ and $n \in \{1,\dots, N\}$, and performing a discrete integration by parts on the contribution of the fluxes, namely $ \sum_{K\in\Tt} \sum_{\sigma\in\Ee_K} F_{i,K\sigma}^n \varphi_K^n = \sum_{\sig \in \Ee_{\rm int}}  F_{i,K\sig}^n (\varphi_{K\sig}^n - \varphi_K^n)$ for all $i = 1, \dots, I$, provides
\be\label{eq:L2H-1.1}
\sum_{n=1}^N \sum_{K\in\Tt} m_K (u_{i,K}^n - u_{i,K}^{n-1}) \varphi_K^n = \sum_{n=1}^N \tau^n \sum_{\sig \in \Ee_{\rm int}} F_{i,K\sig}^n (\varphi_{K\sig}^n - \varphi_K^n). 
\ee
Applying Cauchy--Schwarz inequality to the right-hand side then using Lemma~\ref{lem:FiL2} provides the desired result. 
The recovery of the estimate for $i=0$ then directly follows from the definition~\eqref{eq:u0Kn} of $u_{0,K}^n$ and from \eqref{eq:L2H-1} for $i \geq 1$.
\end{proof}

\section{Convergence of the schemes}\label{sec:conv}
This section is devoted to the proof of Theorem~\ref{theo:main.2}, which relies on compactness arguments. 
Given sequences $\left( \mathcal{T}_\ell, \mathcal{E}_\ell, \lbrace{x}_K\rbrace_{K \in \mathcal{T}_\ell}\right)_{\ell \geq 1}$ and $\left(\tau_\ell\right)_{\ell\geq 1}=\left(\left(\tau^n_\ell\right)_{n\geq 1}\right)_{\ell\geq 1}$ 
of admissible meshes and a sequence of time steps in the sense of Section~\ref{sec:space_time_discr}, with 
\be\label{eq:mesh.ell}
\lim_{\ell \to \infty} h_{\Tt_\ell} =\lim_{\ell \to \infty} \Delta t_{\ell}  = 0 \quad \text{and}\quad \zeta_{\Tt_\ell}\leq \zeta_{\star} < + \infty, 
\ee
we define the sequences $(u_{i,\Tt_\ell, \tau_\ell})_{\ell \geq 1}$ and $(\phi_{\Tt_\ell, \tau_\ell})_{\ell \geq 1}$ as in \eqref{eq:uTt-phiTt}\&\eqref{eq:u0Tt}. For readability, the $\ell$ index is removed when not essential for understanding. 

As usual in the analysis of finite volume schemes, we also need to handle quantities attached to the faces $\sig \in \Ee$. To this end, we introduce the so-called diamond mesh of $\O$ by associating a diamond cell $\omega_\sig$ to every $\sig \in \Ee$. More precisely, $\omega_\sig$ is the convex hull of $\{x_K, x_L, \sigma\}$ if $\sigma = K|L \in \Ee_{\rm int}$, and 
$\omega_\sig = \operatorname{conv}\{x_K, \sigma\}$ if $\sigma \in \Ee_\text{ext} \cap \Ee_K$, see Figure~\ref{fig:diamond_cell} for an illustration. 
The Lebesgue measure $m_{\omega_\sig}$ of $\omega_\sig$ is then given by 
\be\label{eq:m_omega_sig}
m_{\omega_\sig} = \frac{m_\sig d_\sig}d, \qquad \sig \in \Ee,
\ee
where $d$ is the dimension of the space domain $\O$.

Among other quantities attached to faces, one defines the inflated fluxes 
$\left(F_{i,\Ee,\tau}\right)_{1 \leq i \leq I}$ as the piecewise constant in space and time vector fields defined by 
\[
F_{i,\Ee,\tau}(t,x) = \frac{d}{m_\sig}\, F_{i,K\sig}^n n_{K\sig} \quad \text{if}\; (t,x) \in (t^{n-1}, t^n] \times \omega_\sig, 
\]
with $n_{K\sig}$ the normal to $\sigma$ outward with respect to $K$.
A major part of the analysis carried out in this section consists in showing that, up to the extraction of a subsequence, 
\[
F_{i,\Ee_\ell,\tau_\ell}(t,x) \underset{\ell \to +\infty}\longrightarrow F_i \quad \text{weakly in}\; L^2_{\rm loc}(\Rplus \times \overline \O)^d. 
\]
where we make use of the definition~\eqref{eq:Fi.3} of the continuous fluxes.

\begin{figure}[ht]
    \centering
    \begin{tikzpicture}[scale = .6]
\def\decalter{.5}
\def\decalbis{0.3}
\draw[line width = 1pt] (0,0)--(2,6)--(6,2)--(0,0);
\draw[line width = 1pt] (2,6)--(6,2)--(8,6)--(2,6);
\draw[color=white, fill = orange!20!white, line width = 0pt] (2.5,2.5)--(2,6)--(5,5)--(6,2)--cycle;
\draw[line width = 1.5pt, color = orange] (2,6)--(6,2);
\draw[color=white, fill = green!20!white, line width = 0pt] (2.5,2.5)--(6,2)--(0,0)--cycle;
\draw[line width = 1.5pt, color = green!50!black] (0,0)--(6,2);

\draw (5,5) node {$\color{red!80!black}\bullet$};
\draw (2.5,2.5) node {$\color{green!60!black}\bullet$};
\draw (2.5,2.4) node[left] {$x_K$};
\draw (5,5.1) node[right] {$x_L$};
\draw (3.15,5.15) node {\rotatebox{-45}{$\color{orange} \sig = K|L$}};
\draw (3.8+\decalbis,3.3 - \decalbis) node {$\color{orange} \omega_{\sig}$};
\draw (3.15,.8) node {\rotatebox{+16}{$\color{green!50!black} \sig' \subset \partial\O$}};
\draw (2.5,2) node[below] {$\color{green!50!black} \omega_{\sig'}$};
\end{tikzpicture}
  \caption{Examples of diamond cells $\omega_\sig, \omega_{\sigma'}$ for inner and external faces $\sigma$ and $\sigma'$.}
\label{fig:diamond_cell}
\end{figure}
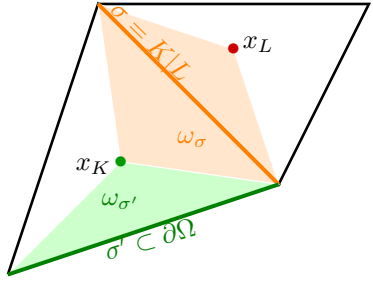

Following~\cite{CHLP03, EG03}, to a piecewise constant in space function $v_\Tt(x) = \sum_{K\in\Tt} v_K \mathbbm{1}_K(x)$, we associate the inflated gradient $\grad_\Tt v_\Tt: \O \to \R^d$ defined by 
\begin{subequations}\label{eq:grad_T}
\be
\grad_\Tt v_\Tt(x) = d\, \frac{v_{K\sig} - v_K}{d_\sig} n_{K\sig} \quad \text{if}\; x \in \omega_\sig.
\ee
This operator straightforwardly extends to piecewise constant functions in space and time:
\be
\grad_\Tt v_{\Tt,\tau}(t,x) = d\, \frac{v_{K\sig}^n - v_K^n}{d_\sig} n_{K\sig} \quad \text{if}\; (t,x) \in (t^{n-1},t^n]\times \omega_\sig.
\ee
\end{subequations}

This section is organized as follows.
The needed compactness properties are established in Section~\ref{ssec:compact}, allowing to claim for the existence of 
(at least) a limit point $(U, \phi)$ to the sequence ${(U_{\Tt_\ell, \tau_\ell}, \phi_{\Tt_m, \Delta t_\ell})}_{\ell \geq 1}$. 
Section~\ref{ssec:identify} concludes the proof of Theorem~\ref{theo:main.2} by showing that the limit value $(U, \phi)$ is a weak solution of the continuous problem \eqref{eq:contmodel_conserv}--\eqref{eq:init} in the sense of Definition~\ref{def:weak}.

\subsection{Compactness of the approximate solutions}\label{ssec:compact}

We establish here enough compactness properties to pass to the limit in the scheme. 

\begin{prop}[Compactness results]\label{prop:compact}
There exist functions $u_0 \in L^{\infty}(\Rplus\times \O)$ with $\sqrt{u_0} \in L^2_{\rm loc}(\Rplus; H^1(\O))$, $u_1, \dots, u_n  \in L^{\infty}(\Rplus\times\O)$ with $u_i \sqrt{u_0} \in L^2_{\rm loc}(\Rplus; H^1(\O))$, and $\phi \in L^{\infty}(\Rplus\times\O)$ with $\phi - \phi^D \in L^\infty(\Rplus; V)$ such that, up to a subsequence, as $\ell \to +\infty$, there holds
\begin{align}
\label{eq:conv.ui}
u_{i,\Tt,\tau} \rightarrow u_i &\quad\text{in the $L^\infty(\Rplus\times\O)$-weak-$\star$ sense}, \; 1 \leq i \leq I, 
\\
\label{eq:conv.u0}
u_{0,\Tt,\tau} \rightarrow u_0 &\quad\text{in the $L^\infty(\Rplus\times\O)$-weak-$\star$ 
sense and a.e. in $\Rplus\times\O$},
\\
\label{eq:conv.phi}
\phi_{\Tt,\tau}  \rightarrow \phi &\quad\text{in the $L^\infty(\Rplus\times\O)$-weak-$\star$ 
sense and a.e. in $\Rplus\times\O$}.
\end{align}
Moreover, for $1 \leq i \leq I$, one has 
\begin{align}
\label{eq:conv.sqrt_u0_ui}
u_{i,\Tt,\tau} \sqrt{u_{0,\Tt,\tau}}   \rightarrow u_i \sqrt{u_0} &\quad\text{a.e. in } \Rplus\times\O, 
\\
\label{eq:conv.u0_ui}
u_{0,\Tt,\tau} u_{i,\Tt,\tau}   \rightarrow u_0 u_i &\quad\text{a.e. in } \Rplus\times\O, 
\end{align}
and also
\begin{align}
\label{eq:conv.sqrt_grad_u0}
\nabla_{\Tt, \tau} \sqrt{u_{0,\Tt,\tau}} \rightharpoonup \nabla \sqrt{u_0} &\quad\text{weakly in } L^2_{{\rm loc}}(\Rplus; L^2(\O)^d), \\
\label{eq:conv.grad_u0_ui}
\nabla_{\Tt, \tau} [u_{0,\Tt,\tau} u_{i,\Tt,\tau}]  \rightharpoonup \nabla [u_0 u_i] &\quad\text{weakly in } L^2_{{\rm loc}}(\Rplus; L^2(\O)),
\\
\label{eq:conv.grad_phi}
\nabla_{\Tt, \tau} \phi_{\Tt,\tau} \rightharpoonup \nabla \phi &\quad\text{weakly in } L^2_{{\rm loc}}(\Rplus; L^2(\O)),
\end{align}
\end{prop}
\begin{proof}
    The $L^\infty(\Rplus\times\O)$-weak-$\star$ convergences \eqref{eq:conv.ui}--\eqref{eq:conv.phi} are the consequences of the uniform bounds on $u_{i,\Tt,\tau}$ and $\phi_{\Tt,\tau}$ in $L^\infty(\Rplus\times\O)$, thanks to the Banach--Alaoglu Theorem.
{The almost everywhere convergence \eqref{eq:conv.u0} follows from strong convergence given by the} Aubin--Lions lemma, see \cite{GL12} for a general presentation of the lemma, and \cite[Lemma 9]{CCGJ19} for its application in our context{. It} can be applied thanks to the uniform estimates of Lemma~\ref{lem:L2locH1} and Corollary \ref{cor:L2H-1}. {These estimates also yield} the weak convergence in $L^2_{{\rm loc}}(\Rplus; L^2(\O))$ of $\nabla_{\Tt, \tau} u_{0,\Tt,\tau}$ towards $\nabla u_0$. Concerning the 
 point-wise convergence of $\phi_{\Tt,\tau}$ towards $\phi$, it can be proven by using the discrete $L^2_{\rm loc}((H^1)')$ estimate on time increments of the right-hand side of the discrete Poisson equation~\eqref{eq:scheme.Poisson} that follows from Corollary~\ref{cor:L2H-1}. As the proof is fully similar to the one of~\cite[Proposition 4.2]{CCFG21}, we do not provide details here. 
The weak compactness property  \eqref{eq:conv.grad_phi} is also established in \cite[Proposition 4.2]{CCFG21}.

In order to prove~\eqref{eq:conv.sqrt_grad_u0}, one can either make use of the non-linear Aubin--Simon theorem of~\cite{ACM17}, or directly remark that, thanks to Lemma~\ref{lem:L2locH1}, 
the vector field $\grad_\Tt \sqrt{u_{0,\Tt,\tau}}$ is uniformly bounded in $L^2_{\rm loc}(\Rplus; L^2(\O)^d)$. We deduce from this boundedness that there exists $G_0 \in L^2_{\rm loc}(\Rplus; L^2(\O)^d)$ such that, up to a subsequence, $\grad_\Tt \sqrt{u_{0,\Tt,\tau}}$ tends to $G_0$ weakly. As $\sqrt{u_{0,\Tt,\tau}}$ converges point-wise (and thus in $L^1_{\rm loc}(\Rplus\times\O)$) towards $\sqrt{u_0}$ owing to~\eqref{eq:conv.u0}, the weak consistency of the inflated gradient $\grad_{\Tt}$, see for instance \cite{CHLP03, DE09}, allows to show that $G_0 = \grad \sqrt{u_0}$.

    Regarding the point-wise convergences \eqref{eq:conv.sqrt_u0_ui} and \eqref{eq:conv.u0_ui}, of $u_{i,\Tt,\tau} \sqrt{u_{0,\Tt,\tau}}$, respectively, we can conclude by applying \cite[Lemma 10.~(Discrete Aubin--Lions of “degenerate” type)]{CCGJ19}. To prove \eqref{eq:conv.sqrt_u0_ui}, we use it with $y_\ell=\sqrt{u_{0,\Tt_\ell,\tau_\ell}}$ and $z_\ell=u_{i,\Tt_\ell,\tau_\ell}$, taking advantage of the uniform estimate \eqref{eq:L2H-1}, \eqref{eq:L2locH1.1}, and \eqref{eq:L2locH1.2}. Whereas, with $y_\ell=u_{0,\Tt_\ell,\tau_\ell}$ and $z_\ell=u_{i,\Tt_\ell,\tau_\ell}$, we conclude \eqref{eq:conv.u0_ui} similarly, replacing \eqref{eq:L2locH1.2} by \eqref{eq:L2locH1.4}.

It finally remains to prove \eqref{eq:conv.grad_u0_ui} for $i=1,\dots, I$. As a consequence of~\eqref{eq:L2locH1.4}, $\grad_\Tt [u_{i,\Tt,\tau} u_{0,\Tt,\tau}]$ is uniformly bounded in $L^2_{\rm loc}(\Rplus\times \overline \O)^d$. Hence there exists $G_i$ such that, up to a subsequence, 
\[
\grad_\Tt [u_{i,\Tt,\tau} u_{0,\Tt,\tau}] \rightharpoonup G_i \quad \text{weakly in }L^2_{\rm loc}(\Rplus\times \overline \O)^d.
\]
As $u_{i,\Tt,\tau} u_{0,\Tt,\tau}$ converges towards $u_i u_0$ thanks to~\eqref{eq:conv.u0_ui}, we can invoke again the weak consistency of the inflated gradient to conclude that $G_i = \grad [u_i u_0]$.
\end{proof}

\subsection{Identification of the limit}\label{ssec:identify}
In this section, we conclude the proof of Theorem~\ref{theo:main.2} by showing the following proposition. 
\begin{prop}\label{prop:identify}
Let $(U,\phi)$ be as in Proposition~\ref{prop:compact}, then it is a weak solution of the problem~\eqref{eq:contmodel_conserv}--\eqref{eq:init} in the sense of Definition~\ref{def:weak}.
\end{prop}
\begin{proof}
The regularity requirements on $U$ and $\phi$ have already been checked in Proposition~\ref{prop:compact}. Therefore it only remains to verify that the weak formulations~\eqref{eq:weak.phi} and \eqref{eq:weak.ui} hold true. The case of the Poisson equation is classical. It will not be detailed here, and we refer to~\cite[Proposition 4.2]{CCFG21} for a synthetic proof. We rather focus our attention on the derivation of~\eqref{eq:weak.ui}.

Let $\varphi \in C^\infty_c(\Rplus\times\overline\Omega)$, then, for some admissible mesh $(\Tt_\ell, \Ee_\ell, (x_K)_{K\in\Tt_\ell})$ (we remove the subscript $\ell$ when possible for legibility), denote by $\varphi_K^n = \varphi(t^n, x_K)$. 
Multiplying~\eqref{eq:fvscheme_conserv} by $\tau^n \varphi_K^{n-1}$ and summing over $n \geq 1$ and $K \in \Tt$ yields, as for~\eqref{eq:L2H-1.1}:
\be\label{eq:identify.1}
\sum_{n=1}^N \sum_{K\in\Tt}m_K  (u_{i,K}^n - u_{i,K}^{n-1}) \varphi_K^{n-1} = \sum_{n=1}^N \tau^n \sum_{\sig \in \Ee_{\rm int}} F_{i,K\sig}^n (\varphi_{K\sig}^{n-1} - \varphi_K^{n-1}). 
\ee
Since $\varphi$ is compactly supported in time, $ \varphi_K^{n} = 0$ for $n$ large enough. Therefore, the left-hand side of~\eqref{eq:identify.1} is rewritten as
\[
\sum_{n=1}^N \sum_{K\in\Tt}m_K  (u_{i,K}^n - u_{i,K}^{n-1}) \varphi_K^{n-1}= \sum_{n=1}^N \tau^n \sum_{K\in\Tt} m_K u_{i,K}^n \frac{\varphi_K^{n-1} - \varphi_K^n}{\tau^n} - 
\sum_{K\in\Tt} m_K u_{i,K}^0 \varphi_K^0.
\]
Since $\varphi$ is smooth, the approximate time derivative $\delta_\tau \varphi_{\Tt,\tau}$ of $\varphi$ defined by 
\[
\delta_\tau \varphi_{\Tt,\tau}(t,x) =  \frac{\varphi_K^{n} - \varphi_K^{n-1}}{\tau^n} \quad \text{if}\; (t,x) \in (t^{n-1},t^n] \times K
\]
converges in $L^1(\Rplus\times\O)$ towards $\partial_t \varphi$. Therefore, using~\eqref{eq:conv.ui}, we get that 
\be\label{eq:identify.2}
\sum_{n=1}^N \tau^n \sum_{K\in\Tt} m_K u_{i,K}^n \frac{\varphi_K^{n-1} - \varphi_K^n}{\tau^n} \underset{\ell\to+\infty} \longrightarrow - \iint_{\Rplus\times\O} u_i \partial_t \varphi. 
\ee
The function $u_{i,\Tt}^0 = \sum_{K\in\Tt} u_{i,K}^0 \mathbbm1_K$ converges strongly in $L^1(\O)$ towards $u_i^0$, while $\varphi_\Tt^0= \sum_{K\in\Tt} \varphi_K^0 \mathbbm1_K$ converges uniformly towards $\varphi(0,\cdot)$ thanks to the regularity of $\varphi$. Therefore, 
\be\label{eq:identify.3}
-\sum_{K\in\Tt} m_K u_{i,K}^0 \varphi_K^0 \underset{\ell\to+\infty} \longrightarrow - \int_\O u_i^0 \varphi(0,\cdot), 
\ee
so we can pass to the limit in the left-hand side of~\eqref{eq:identify.1}.

Let us now focus on its right-hand side. 
Define $\tilde \varphi_{\Tt,\tau}$ by 
\[
\tilde \varphi_{\Tt,\tau}(t,x) = \varphi(t^{n-1},x_K) \quad \text{if\;} (t,x) \in [t^{n-1},t^n) \times K. 
\]
 Following~\cite{DE_FVCA8} (see~\cite{CVV99} for a practical example), 
one can reconstruct a second approximate gradient operator $\widehat \grad_{\Tt}$ mapping the set of piecewise constant functions in space and time to $\R^d$ such that 
\[
\widehat \grad_\Tt \tilde\varphi_{\Tt,\tau}(t,x) \cdot n_{K\sig} =
\frac{\varphi_{L}^n - \varphi_K^n}{d_\sig} \quad \text{if}\; (t,x) \in [t^{n-1}, t^n) \times \omega_\sig, \quad \sig = K|L \in \Ee_{\rm int}, 
\]
and which is strongly consistent, i.e., 
\[
\widehat \grad_\Tt \tilde\varphi_{\Tt,\tau} \underset{\ell \to +\infty} \longrightarrow \grad \varphi \quad \text{strongly in }L^2(\Rplus\times \O)^d.
\]
The right-hand side of~\eqref{eq:identify.1} then simply boils down to 
\[
 \sum_{n=1}^N \tau^n \sum_{\sig \in \Ee_{\rm int}} F_{i,K\sig}^n (\varphi_{K\sig}^{n-1} - \varphi_K^{n-1}) 
 = \iint_{\Rplus\times\O} F_{i,\Ee,\tau}\cdot \widehat \grad_\Tt \tilde\varphi_{\Tt,\tau}.
\]
We have shown in Lemma~\ref{lem:FiL2} that $\left(F_{i,\Ee_\ell,\tau_\ell}\right)_{\ell\geq 1}$ is uniformly bounded in $L^2_{\rm loc}(\Rplus\times \overline \O)^d$. Therefore, it converges (up to a subsequence) towards some vector field 
$F_i$ weakly in $L^2_{\rm loc}(\Rplus\times \overline \O)$. We can then pass to the limit in the right-hand side of~\eqref{eq:identify.1} to obtain 
\be\label{eq:identify.4}
 \sum_{n=1}^N \tau^n \sum_{\sig \in \Ee_{\rm int}} F_{i,K\sig}^n (\varphi_{K\sig}^{n-1} - \varphi_K^{n-1}) 
  \underset{\ell \to +\infty} \longrightarrow \iint_{\Rplus\times\O} F_i \cdot \grad \varphi.
\ee

To conclude the proof of Proposition~\ref{prop:identify}, we still have to identify the limiting flux $F_i$ under the 
form~\eqref{eq:Fi.3}. To this end, we split the inflated fluxes into a convective and diffusive part, 
\begin{align*}
F_{i,\Ee,\tau}^\text{conv}(t,x) =& \frac{d}{m_\sig}\, F_{i,K\sig}^{\text{conv}, n} n_{K\sig} \quad \text{if}\; (t,x) \in (t^{n-1}, t^n] \times \omega_\sig, \\
F_{i,\Ee,\tau}^\text{diff}(t,x) =& \frac{d}{m_\sig}\, F_{i,K\sig}^{\text{diff}, n} n_{K\sig} \quad \text{if}\; (t,x) \in (t^{n-1}, t^n] \times \omega_\sig, 
\end{align*}
where $F_{i,K\sig}^{\text{conv}, n}$ and $F_{i,K\sig}^{\text{diff}, n}$ are as in the proof of Lemma~\ref{lem:FiL2}. 
As both $F_{i,\Ee,\tau}^\text{conv}$ and {$F_{i,\Ee,\tau}^\text{diff}$} have been shown to be bounded in $L^2_{\rm loc}(\Rplus\times \overline \O)^d$, one has, up to a subsequence,
\be\label{eq:identify.5}
F_{i,\Ee,\tau}^\text{diff} \underset{\ell \to +\infty} \longrightarrow F_i^\text{diff} \quad\text{and}\quad F_{i,\Ee,\tau}^\text{conv}\underset{\ell \to +\infty} \longrightarrow F_i^\text{conv} \quad \text{weakly in }L^2_{\rm loc}(\Rplus\times \overline \O)^d.
\ee  

Let us first focus on the convective part $F_{i,\Ee,\tau}^\text{conv}$. For $i=1,\dots, I$, $\sig \in \Ee$ and $n \geq 1$, define 
\[
\eta_{i,\sig}^n = \begin{cases}
\frac{u_{i,K}^n u_{0,L}^n + u_{i,L}^n u_{0,K}^n}2 & \text{if}\; \sig = K|L \in \Ee_{\rm int}, \\
u_{i,K}^n u_{0,K}^n & \text{if}\; \sig \in \Ee_K \cap \Ee_\text{ext}, 
\end{cases}
\]
and
\[
\tilde \eta_{i,\sig}^n = \begin{cases}
\frac{u_{i,K}^n u_{0,K}^n + u_{i,L}^n u_{0,L}^n}2 & \text{if}\; \sig = K|L \in \Ee_{\rm int}, \\
u_{i,K}^n u_{0,K}^n & \text{if}\; \sig \in \Ee_K \cap \Ee_\text{ext}, 
\end{cases}
\]
and then
$\eta_{i,\Ee,\tau}(t,x) = \eta_{i,\sig}^n$ and $\tilde \eta_{i,\Ee,\tau}(t,x) = \eta_{i,\sig}^n$ if $(t,x) \in (t^{n-1},t^n] \times \omega_\sig$. In view of~\eqref{eq:Fi.conv} and of the definition~\eqref{eq:grad_T} of the inflated gradient, the convective part of the inflated approximate flux is rewritten as
\be\label{eq:Fi_Ee.conv}
F_{i,\Ee,\tau}^\text{conv} = - D_i\, \eta_{i,\Ee,\tau}\, z_i \, \grad_\Tt \phi_{\Tt,\tau}.
\ee
 Thanks to~\eqref{eq:conv.u0_ui}, proving that 
\be\label{eq:conv.etai.1}
\eta_{i,\Ee,\tau} - \tilde \eta_{i,\Ee,\tau} \underset{\ell \to +\infty} \longrightarrow 0 \quad \text{a.e. in } \Rplus \times \O
\ee
and
\be\label{eq:conv.etai.2}
\tilde \eta_{i,\Ee,\tau} - u_{0,\Tt,\tau} u_{i,\Tt,\tau} \underset{\ell \to +\infty} \longrightarrow 0 \quad \text{a.e. in } \Rplus \times \O
\ee
is equivalent to showing that 
\[
\eta_{i,\Ee,\tau}\underset{\ell \to +\infty} \longrightarrow u_0 u_i  \quad \text{a.e. in } \Rplus \times \O, 
\]
and thus strongly in $L^2_{\rm loc}(\Rplus \times \overline \O)$ thanks to the uniform bound $0 \leq u_{i,\Tt,\tau} \leq 1$. 
Bearing in mind~\eqref{eq:conv.grad_phi}, this allows to pass to the weak limit in \eqref{eq:Fi_Ee.conv} and to recover that 
\be\label{eq:identify.Fi.conv}
F_i^\text{conv} = - D_i\, u_0u_i\, z_i\, \grad \phi. 
\ee

To establish~\eqref{eq:conv.etai.1}, remark that 
\[
\left\|\eta_{i,\Ee,\tau} - \tilde \eta_{i,\Ee,\tau}\right\|_{L^2((0,t^N)\times\O)} ^2 
= \frac14\sum_{n=1}^N \tau^n \sum_{\sig \in \Ee_{\rm int}} m_{\omega_\sig} (u_{i,K}^n - u_{i,L}^n)^2  (u_{0,K}^n - u_{0,L}^n)^2. 
\]
Since $0 \leq u_{i,K}^n \leq 1$, owing to~\eqref{eq:m_omega_sig}, and since $d_\sig \leq 2 h_\Tt$ this gives 
\[
\left\|\eta_{i,\Ee,\tau} - \tilde \eta_{i,\Ee,\tau}\right\|_{L^2((0,t^N)\times\O)}^2 \leq \frac{h_\Tt^2}d 
\sum_{n=1}^N \tau^n \sum_{\sig \in \Ee_{\rm int}} a_\sig (u_{0,K}^n - u_{0,L}^n)^2. 
\]
We can make use of estimate~\eqref{eq:L2locH1.1} to get that $\eta_{i,\Ee,\tau} - \tilde \eta_{i,\Ee,\tau}$ tends to $0$ in $L^2_{\rm loc}(\Rplus\times\overline \O)$, hence almost everywhere up to the extraction of yet another subsequence, 
whence~\eqref{eq:conv.etai.1}. 

Concerning~\eqref{eq:conv.etai.2}, we can proceed similarly to obtain 
\[
\left\|\tilde \eta_{i,\Ee,\tau} - u_{0,\Tt,\tau} u_{i,\Tt,\tau} \right\|_{L^2((0,t^N)\times\O)} ^2 
{\leq}  \frac{h_\Tt^2}d \sum_{n=1}^N \tau^n \sum_{\sig \in \Ee_{\rm int}} a_\sig \left( u_{0,K}^n u_{i,K}^n -  u_{0,L}^n u_{i,L}^n \right)^2
\]
which tends to $0$ thanks to~\eqref{eq:L2locH1.4}. 

It finally remains to identify the limit of $F_{i,\Ee,\tau}^\text{diff}$ as $F_i^\text{diff}$. To this end, remark first that, for $\sig = K|L \in \Ee_{\rm int}$ and $n \geq 1$, one has 
\[
u_{i,K}^n u_{0,L}^n - u_{i,L}^n u_{0,K}^n = T_{1,\sig}^n + T_{2,\sig}^n + T_{3,\sig}^n
\]
with 
\begin{align*}
 T_{1,\sig}^n =\; & u_{i,K}^n u_{0,K}^n - u_{i,L}^n u_{0,L}^n,\\
 T_{2,\sig}^n =\; & -2 \left(u_{i,K}^n \sqrt{u_{0,K}^n} + u_{i,L}^n \sqrt{u_{0,L}^n}\right)\left( \sqrt{u_{0,K}^n} - \sqrt{u_{0,L}^n}\right) ,\\
 T_{3,\sig}^n =\; & \left(u_{i,K}^n - u_{i,L}^n\right) \left(\sqrt{u_{0,K}^n} - \sqrt{u_{0,L}^n}\right)^2.
\end{align*}
Therefore, in view of~\eqref{eq:Fi.diff}, we can split 
\be\label{eq:identify.Fi_diff}
F_{i,\Ee,\tau}^\text{diff} = - D_i \grad_\Tt [u_{i,\Tt,\tau} u_{0,\Tt,\tau}]  + 4 D_i \gamma_{i,\Ee,\tau} \grad_\Tt \sqrt{u_{0,\Tt,\tau}} + \Rr_{i,\Ee,\tau} + \Ss_{i,\Ee,\tau} ,
\ee
where we have set 
\[
\gamma_{i,\Ee,\tau}(t,x) = \begin{cases}
\frac12\left(u_{i,K}^n \sqrt{u_{0,K}^n} + u_{i,L}^n \sqrt{u_{0,L}^n}\right) & \text{if}\; (t,x) \in (t^{n-1},t^n] \times \omega_{\sig}, \; \sig = K|L \in \Ee_{\rm int}, \\
u_{i,K}^n \sqrt{u_{0,K}^n} &  \text{if}\; (t,x) \in (t^{n-1},t^n] \times \omega_{\sig}, \; \sig \in \Ee_K \cap \Ee_{\text{ext}}, 
\end{cases} 
\]
and where, for $(t,x) \in (t^{n-1},t^n] \times \omega_{\sig}$,  $\sig = K|L \in \Ee_{\rm int},$
\[
\Rr_{i,\Ee,\tau}(t,x) = d \, D_i \, (u_{i,K}^n - u_{i,L}^n)  \left(\sqrt{u_{0,K}^n} - \sqrt{u_{0,L}^n}\right)^2 n_{K\sig}, 
\]
and, in view of \eqref{eq:Fi.diff}, 
\[
\left| \Ss_{i,\Ee,\tau}(t,x) \right| \leq \frac{C}{d_\sig} \left(\phi_K^n - \phi_L^n\right)^2. 
\]
Both $\Rr_{i,\Ee,\tau}$ and $\Ss_{i,\Ee,\tau}$ are assumed to vanish on $\Rplus \times \omega_\sig$ for all $\sig \in \Ee_\text{ext}$.
Then thanks to Proposition~\ref{prop:bound.phi}, one has 
\[
\|\Ss_{i,\Ee,\tau}\|_{L^1((0,t^N)\times\O)^d} \leq C \sum_{n=1}^N \tau^n \sum_{\sig \in \Ee_{\rm int}} m_\sig \left(\phi_K^n - \phi_L^n\right)^2 \leq C h_\Tt t^N.
\]
On the other hand, since $0 \leq u_{i,K}^n \leq 1$, we deduce from~\eqref{eq:L2locH1.1} that 
\[
\|\Rr_{i,\Ee,\tau}\|_{L^1((0,t^N)\times\O)^d} \leq D_i \sum_{n=1}^N \tau^n \sum_{\sig \in \Ee_{\rm int}} m_\sig
\left(\sqrt{u_{0,K}^n} - \sqrt{u_{0,L}^n}\right)^2 \leq C h_\Tt (1+t^N). 
\]
The two last terms in~\eqref{eq:identify.Fi_diff} then tend to $0$, while the first term tends to $-D_i \grad [u_0u_i]$ thanks to 
\eqref{eq:conv.grad_u0_ui}. In view of the weak $L^2_{\rm loc}(L^2)$ convergence of $\grad_\Tt \sqrt{u_{0,\Tt,\tau}}$ towards $\grad \sqrt{u_0}$, cf. \eqref{eq:conv.sqrt_grad_u0}, then it suffices to show that 
\[
\gamma_{i,\Ee,\tau} \underset{\ell \to+\infty} \longrightarrow u_i \sqrt{u_0} \quad \text{strongly in\; } L^2_{\rm loc}(\Rplus \times \overline \O) 
\]
to pass to the limit in~\eqref{eq:identify.Fi_diff}. As $u_{i,\Tt,\tau} \sqrt{u_{0,\Tt,\tau}}$ converges strongly towards $u_i \sqrt{u_0}$, see~\eqref{eq:conv.sqrt_u0_ui}, then 
\begin{align*}
\|\gamma_{i,\Ee,\tau} - u_{i,\Tt,\tau} \sqrt{u_{0,\Tt,\tau}}\|_{L^2((0,t^N)\times \O)}^2 =\; & \frac12 \sum_{n=1}^N \tau^n \sum_{\sig \in \Ee_{\rm int}} m_{\omega_\sig} \left( u_{i,K}^n \sqrt{u_{0,K}^n} - u_{i,L}^n \sqrt{u_{0,L}^n} \right)^2 \\
\leq \; & \frac{h_\Tt^2}d \sum_{n=1}^N \tau^n \sum_{\sig \in \Ee_{\rm int}} a_\sig \left( u_{i,K}^n \sqrt{u_{0,K}^n} - u_{i,L}^n \sqrt{u_{0,L}^n} \right)^2 \\
\overset{\eqref{eq:L2locH1.2}}\leq \; & \frac{h_\Tt^2}d (1+t^N)  \underset{\ell \to+\infty} \longrightarrow 0.
\end{align*}
This completes the proof of~\eqref{eq:identify.5}, thus the ones of Proposition \ref{prop:identify} and Theorem~\ref{theo:main.2}.
\end{proof}

\section{Long-time asymptotic of the scheme}\label{sec:asymptotic}

Our aim here is to prove our last main theoretical contribution, namely Theorem~\ref{theo:main.3}. 
Throughout this section, we will work on a fixed grid $(\Ee,\Tt,(x_K)_{K\in\Tt})$ and for a fixed time discretization. 

We are interested in steady states $(U^\infty_\Tt, \phi_\Tt^\infty)$ of the scheme~\eqref{eq:scheme.Poisson}--\eqref{eq:u0Kn}, i.e. solutions to 
\be\label{eq:scheme.Poisson.inf}
\lambda^2 \sum_{\sig \in \Ee_K} a_\sig( \phi_K^\infty - \phi_{K\sig}^\infty) = m_K\left( f_K + \sum_{i=1}^I z_i u_{i,K}^\infty \right), \qquad K\in\Tt,
\ee
\be\label{eq:scheme.cons.inf}
\sum_{\sig \in \Ee_K} F_{i,K\sig}^\infty = 0, \qquad i = 1,\dots, I, \; K \in \Tt, 
\ee
with $F_{i,K\sigma}^\infty = 0$ if $\sig \in \Ee_\text{ext}$ and
    \begin{equation}\label{eq:scheme.Fi.inf}
    F_{i,K\sigma}^\infty = a_\sig D_i \left(u^\infty_{i,K}u^{\infty}_{0,L} \B\left(z_i(\phi_L^\infty - \phi_K^\infty)\right)
   - u^\infty_{i,L}u^{\infty}_{0,K}\B\left(z_i(\phi^\infty_{K}-\phi^\infty_{L})\right)\right)
    \end{equation}
for $\sig = K|L \in \Ee_{\rm int}$, where we have set
    \begin{equation}\label{eq:u0Kinf}
    u_{0,K}^\infty = 1 - \sum_{i = 1}^I u_{i,K}^\infty, \qquad K \in \Tt. 
    \end{equation}
    The above equations have to be complemented with some discrete counterpart of the constraint on the mass~\eqref{eq:cont.steadymass}, that is 
        \be\label{eq:constraint.mass}
    \sum_{K\in \Tt} m_K u_{i,K}^\infty = \int_\O u_i^0, \qquad i = 1,\dots, I. 
    \ee
We are moreover interested in the convergence of $(U^n_\Tt, \phi_\Tt^n)$ towards $(U^\infty_\Tt, \phi_\Tt^\infty)$ as $n$ goes to $+\infty$.

\begin{prop}\label{prop:steady.1}
There exists a solution to~\eqref{eq:scheme.Poisson.inf}--\eqref{eq:constraint.mass}, with constant in space potentials in  the sense that there exists $\boldsymbol{\mu}_\Tt^\infty = \left(\mu_{i,\Tt}^\infty\right)_{1\leq i \leq I} \in \R^I$ such that
\be\label{eq:steady.mu}
\log\frac{u_{i,K}^\infty}{u_{0,K}^\infty} + z_i \phi_K^\infty = \mu_{i,\Tt}^\infty, \qquad K \in \Tt, \; 1 \leq i \leq I, 
\ee
such that 
\be
\label{eq:steady.conv}
(U_\Tt^n, \phi_\Tt^n) \underset{n\to+\infty}\longrightarrow (U_\Tt^\infty, \phi_\Tt^\infty).
\ee
\end{prop}

\begin{proof}Owing to \eqref{eq:NRG.1} and the non-negativity of the discrete dissipations $\Dd^n_\Tt$, the sequence ${(\Hh_\Tt^n)}_{n\geq 0}$ is decreasing, and moreover bounded thanks to Lemma~\ref{lem:boundentropy}. Therefore, it converges towards some finite limit $\Hh_\Tt^\infty$, while $\Dd_\Tt^n$ has to tend to $0$ as $n$ goes to $+\infty$. Then so does $\Dd_{i,\sig}^n$, cf.~\eqref{eq:Disign}, for all $i=1,\dots, I$ and all $\sig = K|L \in \Ee_{\rm int}$. 
Since $\left(U_\Tt^n, \phi_\Tt^n\right)_{n\geq 0}$ is included in a bounded (thus relatively compact) set, the LaSalle invariance principle yields the existence of (at least) an accumulation point 
$\left(U_\Tt^\infty, \phi_\Tt^\infty\right)$, the energy of which being equal to $\Hh_\Tt^\infty$. 
Then it follows from the discrete energy / energy dissipation inequality~\eqref{eq:NRG.1} that the dissipation 
$\Dd_\Tt^\infty$ corresponding to $\left(U_\Tt^\infty, \phi_\Tt^\infty\right)$ is equal to $0$, then so are all the edge dissipation contributions $\Dd_{i,\sig}^\infty$ which are defined by \eqref{eq:Disign} but with $\left(U_\Tt^\infty, \phi_\Tt^\infty\right)$ instead of $\left(U_\Tt^n, \phi_\Tt^n\right)$. Since $\Dd_{i,\sig}^\infty = 0$ iff the corresponding flux 
$F_{i,K\sig}^\infty$ defined by~\eqref{eq:scheme.Fi.inf}
vanishes for all $i=1,\dots, I$ and $\sig = K|L \in \Ee_{\rm int}$, one gets that $\left(U_\Tt^\infty, \phi_\Tt^\infty\right)$ is a steady solution, and that the convergence \eqref{eq:steady.conv} holds true.

Let us now show \eqref{eq:steady.mu}. The condition $F_{i,K\sig}^\infty= 0$ for $\sig = K|L$ implies the following alternative:
\begin{enumerate}[(i)]
\item\label{item:00} $u_{0,K}^\infty = 0$ and $u_{0,L}^\infty=0$,
\item\label{item:ii} $u_{i,K}^\infty= 0$ and $u_{i,L}^\infty=0$,
\item\label{item:i0} $u_{i,K}^\infty = 0$ and $u_{0,K}^\infty=0$, or $u_{i,L}^\infty = 0$ and $u_{0,L}^\infty=0$, 
\item\label{item:mu} $\log\frac{u_{i,K}^\infty}{u_{0,K}^\infty} + z_i \phi_K^\infty = \log\frac{u_{i,L}^\infty}{u_{0,L}^\infty} + z_i \phi_L^\infty$.
\end{enumerate}
{This can be seen using the two formulations of the fluxes \eqref{eq:scheme.Fi.inf} and \eqref{eq:scheme.Fi.2} at the limit $n\to\infty$. The formula \eqref{eq:scheme.Fi.inf} clearly yields (i) to (iii) if any of the involved volume fractions vanish. If they are all positive then \eqref{eq:scheme.Fi.2} is well-defined at the limit $n\to\infty$  and yields (iv). }

Assume for contradiction that there exists $K\in\Tt$ such that $u_{0,K}^\infty = 0$, then $F_{i,K\sig}^\infty = 0$ for all $i$ requires that either $u_{i,K}^\infty = 0$ for all $1 \leq i \leq I$, which is incompatible with \eqref{eq:u0Kinf}, or that 
$u_{0,L}^\infty = 0$ for all neighboring cells $L$ sharing an edge with $K$. An induction shows relying on the connected character of $\O$ shows that $u_{0,\Tt}^\infty = 0$, which contradicts~\eqref{eq:constraint.mass}. The cases \eqref{item:00} and \eqref{item:i0} are therefore impossible, and $u_{0,\Tt}^\infty > 0$.

Assume now that $u_{i,K}^\infty = 0$ for some $K\in\Tt$, then as $u_{0,K}^\infty>0$, a similar reasoning shows that $u_{i,L}^\infty$ also has to be equal to $0$ for all neighboring cell of $K$. Then $u_{i,\Tt}^\infty = 0$, contradicting again \eqref{eq:constraint.mass}. Therefore, alternative~\eqref{item:ii} is not feasible, and \eqref{item:mu} has to hold true. 
This concludes the proof of Proposition~\ref{prop:steady.1}.
\end{proof}

Given $\phi_\Tt^\infty$ and $\bmu_\Tt^\infty$, one can invert the system \eqref{eq:u0Kinf}\&\eqref{eq:steady.mu}. As in the continuous case, one then gets that 
\[
u_{i,K}^\infty = v_i(\phi_K^\infty, \bmu_\Tt^\infty), \qquad K \in \Tt, \; 0 \leq i \leq I
\]
with $v_i$ defined in~\eqref{eq:vi}.
Incorporating this relation in~\eqref{eq:scheme.Poisson.inf} provides the following discrete counterpart to the modified Poisson--Boltzmann equation~\eqref{eq:cont.PoissonBoltzmann}:
\be\label{eq:scheme.PoissonBoltzmann}
\lambda^2 \sum_{\sig \in \Ee_K} a_\sig( \phi_K^\infty - \phi_{K\sig}^\infty) + m_K r(\phi_K^\infty, \bmu_\Tt^\infty) = m_K f_K, \qquad K\in\Tt,
\ee
to be complemented with the mass constraints \eqref{eq:constraint.mass}, with $r$ being the non-decreasing function defined in~\eqref{eq:r}.

\begin{prop}\label{prop:steady.2}
The solution $(\phi_\Tt^\infty, \bmu_\Tt^\infty)$ to~\eqref{eq:scheme.PoissonBoltzmann}\&\eqref{eq:constraint.mass} minimizes the strictly convex functional $\Psi_\Tt$ defined by~\eqref{eq:PsiTt}. In particular, $(\phi_\Tt^\infty, \bmu_\Tt^\infty)$ is unique. 
\end{prop}
\begin{proof}
One readily checks that 
\[
\frac{\partial \Psi_\Tt}{\partial {y}_K}({y}_\Tt, \boldsymbol \xi) = \lambda^2 \sum_{\sig \in \Ee_K} a_\sig( {y}_K - {y}_{K\sig}) + m_K r({y}_K, \boldsymbol \xi) - m_K f_K, 
\qquad K \in \Tt. 
\]
and that 
\[
\frac{\partial \Psi_\Tt}{\partial \xi_i }({y}_\Tt, \boldsymbol \xi) = 
\sum_{K\in\Tt} m_K [v_i({y}_K, \boldsymbol \xi) - u_{i,K}^0].
\]
Bearing in mind the definition~\eqref{eq:uiK0} of $u_{i,K}^0$, one gets that $(\phi_\Tt^\infty, \mu_\Tt^\infty)$ solves~\eqref{eq:scheme.PoissonBoltzmann}\&\eqref{eq:constraint.mass} if and only if it is a critical point of $\Psi_\Tt$.  In particular, Proposition~\ref{prop:steady.1} guarantees the existence of such a critical point as it ensures the existence of a solution to ~\eqref{eq:scheme.PoissonBoltzmann}\&\eqref{eq:constraint.mass}. 

To conclude the proof of Proposition~\ref{prop:steady.2}, it remains to show the strict convexity of $\Psi_\Tt$. 
Let $({y}_\Tt,\bxi_\Tt)$ and  $({\tilde y}_\Tt, \tilde \bxi_\Tt)$ be two elements of $\R^\Tt \times \R^I$. Then 
\be\label{eq:poipoi}
\left(D\Psi_\Tt({y}_\Tt,\bxi_\Tt) - D\Psi_\Tt({\tilde y}_\Tt, \tilde \bxi_\Tt) \right) \cdot ({y}_\Tt-{\tilde y}_\Tt,\bxi_\Tt - \tilde \bxi_\Tt) = A+B, 
\ee
where 
\begin{align*}
A =\; & \lambda^2\sum_{K\in\Tt}({y}_K-{\tilde y}_K) \sum_{\sig \in \Ee_K} a_\sig( ({y}_K-{\tilde y}_K) - ({y}_{K\sig} - {\tilde y}_{K\sig})), \\
B =\; & \sum_{K\in\Tt} m_K [r({y}_K, \boldsymbol \xi) - r({\tilde y}_K, \tilde \bxi )] ({y}_K-{\tilde y}_K)   \\
& +  \sum_{K\in\Tt}  \sum_{i=1}^I m_K[(v_i({y}_K, \boldsymbol \xi)-v_i({\tilde y}_K, \tilde \bxi)](\xi_i - \tilde \xi_i).
\end{align*}
Performing a discrete integration by parts gives 
\[
A =  \lambda^2 \sum_{\sig \in \Ee} a_\sig ( ({y}_K-{\tilde y}_K) - ({y}_{K\sig} - {\tilde y}_{K\sig}))^2 \geq 0. 
\]
As we assumed that {$\Ee_\text{ext}^D$} is not empty, it follows from the discrete Poincaré inequality (see Lemma 9.1 and Remark 9.4 of \cite{Eymard2000}) that 
\be\label{eq:poipoi.1}
\text{$A$ vanishes if and only if ${y}_\Tt = {\tilde y}_\Tt$.}
\ee
Using the relation~\eqref{eq:r} between $r$ and $v_i$, we rewrite the term $B$ as 
\[
B = \sum_{K\in\Tt}  m_K \sum_{i=1}^I \big(v_i({y}_K, \bxi) - v_i({\tilde y}_K, \tilde \bxi)\big) \big(\xi_i - z_i {y}_K - (\tilde \xi_i - z_i {\tilde y}_K)\big).
\]
Denote by 
\[
H^*({R}) = \sup_{U \in \Aa} \left[ \sum_{i=1}^I u_i {r}_i - H(U)\right] = \log\left( 1 + \sum_{i=1}^I e^{{r}_i} \right), \qquad {R} = \left({r}_i\right)_{1\leq i \leq I} \in \R^I, 
\]
the Legendre transform of the mixing entropy functional $H$, cf. \eqref{eq:H}, then, setting 
\[
{R}_K = \left({r}_{i,K}\right)_{1\leq i \leq I} =   {(\xi_i - z_i {y}_K)}_{1\leq i \leq I}, \qquad K \in \Tt, 
\]
the term $B$ is rewritten as 
\be\label{eq:poipoi.2}
B = \sum_{K\in\Tt}  m_K \sum_{i=1}^I\left\langle\frac{\partial H^*}{\partial {r}_i}({R}_K) - \frac{\partial H^*}{\partial {r}_i}({\tilde R}_K), {R}_K-{\tilde R}_K\right\rangle.
\ee
As $H^*$ is a Legendre transform, it is by construction convex and thus $B \geq 0$. {By setting ${r}_{0,K} = {\tilde r}_{0,K} = 0$, the convexity of $H^*$ can also be seen as a consequence of Holder's inequality
\[
\sum_{j=0}^I e^{\theta {r}_{j,K} + (1-\theta){\tilde r}_{j,K}}\leq \left(\sum_{j=0}^I e^{{r}_{j,K}}\right)^\theta\left(\sum_{j=0}^I e^{{\tilde r}_{j,K}}\right)^{1-\theta}
\]
for $\theta\in [0,1]$. The inequality is an equality only if $(e^{{r}_{j,K}})_j$ and $(e^{{\tilde r}_{j,K}})_j$ are linearly dependent. As ${r}_{0,K} = {\tilde r}_{0,K}$ it can only happen if ${R}_K = {\tilde R}_K$. In other words $H^*$ is actually strictly convex.}

It follows from the above discussion that the right-hand side of \eqref{eq:poipoi} vanishes if and only if both $A$ and $B$ vanish. Thanks to  \eqref{eq:poipoi.1}, this requires that ${y}_\Tt = {\tilde y}_\Tt$.
On the other hand because of the strict convexity of $H^*$, the term $B$ \eqref{eq:poipoi.2} vanishes if and only if ${R}_\Tt = {\tilde R}_\Tt$, thus if $\bxi = \tilde \bxi$. Then $\Psi_\Tt$ is strictly convex and its minimizer $(\phi_\Tt^\infty, \bmu_\Tt^\infty)$ is unique. 
\end{proof}
\section{Numerical results}\label{sec:num}
    
The non-linear system corresponding to the scheme is solved thanks to a Newton--Raphson method with stopping criterion 
${\| \mathcal{F}_\Tt^n(  (U_K^n)_{K\in\Tt}, (\phi_K^n)_{K\in\Tt})\|}_\infty \leq 10^{-10}$, the components of $\mathcal{F}_\Tt^n$ 
 being given by the left-hand side of \eqref{eq:fvscheme_conserv}.
\subsection{Convergence under grid refinement}
The goal of our first numerical test is to show our scheme is second order accurate w.r.t. the mesh size. To this end, we consider the one-dimensional domain $\O = (0,1)$, in which $I=2$ different ions evolve, both with the same diffusion coefficient $D_1=D_2 =1$. Their (normalized) charge is set to  $z_1 = 2$ and $z_2 = 1$, yielding repulsive interaction. No background charge is considered, i.e. $f=0$, whereas Dirichlet boundary conditions are imposed for the electric potential on both sides of the interval, that are  $\phi^D(t,0) = 10$ and $\phi^D(t,1) = 0$. We consider a moderately small Debye length $\lambda^2=10^{-2}$. We start at initial time $t=0$ with the following configurations:
$u_1^0(x)= 0.2 + 0.1(x-1)$ and $u_2^0 \equiv 0.4$.

\begin{figure}[htb]
\begin{minipage}{.49\textwidth}
\begin{center}
\begin{tikzpicture}
	\begin{axis}[
	xlabel=Space variable $x$,
	ylabel=Concentration,
	legend style={
		legend pos= north west,
	},
	width=\linewidth]
	
	\addplot[color=blue] table[x=x, y=u1] {results_long/snapshots_SQRA.txt};
	\addlegendentry{$u_1$}
	\addplot[color=red] table[x=x, y=u2] {results_long/snapshots_SQRA.txt};
	\addlegendentry{$u_2$}
	\addplot[color=black] table[x=x, y=u0] {results_long/snapshots_SQRA.txt};
	\addlegendentry{$u_0$}
	
	\addplot[color=blue, dashed] table[x=x, y=u1] {results_long/snapshots_SQRA_5e3.txt};
	\addplot[color=red, dashed] table[x=x, y=u2] {results_long/snapshots_SQRA_5e3.txt};
	\addplot[color=black, dashed] table[x=x, y=u0] {results_long/snapshots_SQRA_5e3.txt};
	\end{axis}
\end{tikzpicture}
\end{center}
\end{minipage}
\begin{minipage}{.49\textwidth}
\begin{center}
\begin{tikzpicture}
	\begin{axis}[
	xlabel=Space variable $x$,
	ylabel=Electric potential,
	legend style={
		legend pos= north east,
	},
	width=\linewidth]
	
	\addplot[color=brown] table[x=x, y=phi] {results_long/snapshots_SQRA.txt};
	\addplot[color=brown, dashed] table[x=x, y=phi] {results_long/snapshots_SQRA_5e3.txt};	
	\addlegendentry{$\phi$}
	\end{axis}
\end{tikzpicture}
\end{center}
\end{minipage}
\caption{Left: Concentration profiles $u_1(T,x), u_2(T,x)$ and $u_0(T,x)$ at time $T = 1$ depicted with solid lines. The corresponding long-time limit $(u_1^\infty, u_2^\infty, u_3^\infty)$ is depicted with dashed lines. Right: Electric potential profile $\phi(T,x)$ at time $T = 1$ (solid line) and in the long-time limit (dashed line). }
\label{fig:potential}
\end{figure}
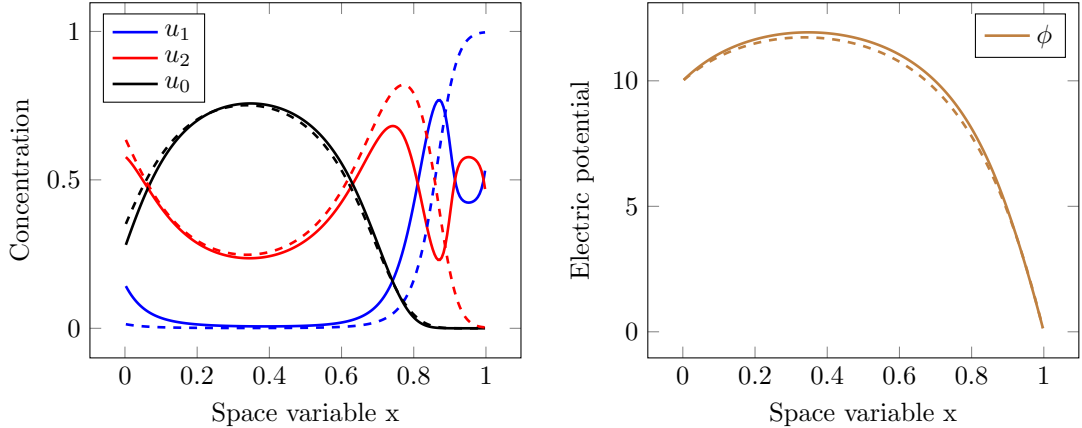

A reference solution is computed on a grid made of $1638400$ cells and with a constant time step $\tau = 10^{-3}$, to which are compared solutions computed on successively refined grids but with the same constant time step. 
The profile of the solution at the final time $T=1$ is depicted on Figures~\ref{fig:potential}. 
The relative space-time $L^1$ error is plotted as a function of the number of cells on Figure~\ref{fig:conv}, showing some second order accuracy in space, as specified in the introductory discussion. 
The number of Newton iterations required to solve a time step remains between 6 for the very first iterations and 2 for larger times and is mainly insensitive to the mesh size. 

\begin{figure}[htb]
\begin{center}
\begin{tikzpicture}
	\begin{loglogaxis}[
	xlabel=number of cells,
	ylabel=relative $L^1_t(L^1_x)$ error,
	legend style={
		legend pos= north east,
	},
	width=0.45\linewidth]
	
	\addplot[mark=square*, color=blue] table[x=NbCells, y=errL1L1] {results_long/conv_SG.txt};
	
	\logLogSlopeTriangle{0.1}{-0.4}{0.1}{2}{black};
	\end{loglogaxis}
\end{tikzpicture}
\end{center}
\caption{Convergence of the schemes under space grid refinement.}
\label{fig:conv}
\end{figure}
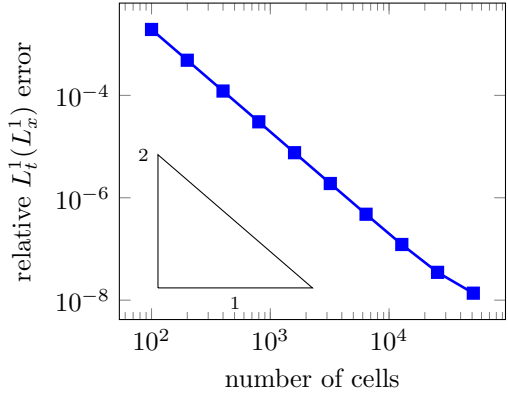

\begin{rema}
An optimal error estimate showing order 2 in space and order 1 in time convergence can be established on cartesian grids if one assumes that there exists a smooth strong solution to the continuous problem, see for instance~\cite{CV23} in a scalar context, or~\cite{DWZ23} for classical Poisson--Nernst--Planck system. The extension to unstructured grids is then much more delicate, as usual for finite volumes, see for instance~\cite{DN18} for a contribution in this direction. However, the existence of such a strong solution to the problem is an open question.
\end{rema}

\subsection{Long-time behavior of the scheme}
We now focus on the long-time behavior of the scheme, and more specifically on the convergence towards its unique steady state. Our study relies on a two-dimensional test case, with $I=3$ charged species (plus the solvent $u_0$). 
The corresponding charges are $z_1 = 2$, $z_2 = 1$, $z_3 = -1$ (and $z_0 = 0$ as the solvent is electrically neutral in our model), while the diffusion coefficients are set to $D_1 = 1$, $D_2 = 2$, and $D_3 = 2$.

The geometry is the following : $\O$ is the unit square $(0,1)^2$, the boundary of which being split into 
$\Gamma^D = \{y=1\} \times \{0 \leq x \leq 1/2\}$ and its complement $\Gamma^N$. The electric potential is set to $0$ on $\Gamma^D$, i.e. $\phi^D = 0$.

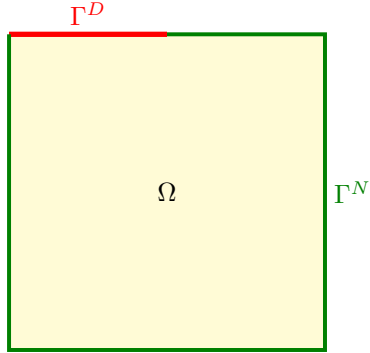
\begin{figure}[htb]
\begin{center}
\begin{tikzpicture}[scale = .7]
\draw[line width = 1pt, color = white, fill = yellow!20!white] (0,0) -- (0,6) -- (6,6) -- (6,0) -- cycle;
\draw[line width = 2pt, color = red] (0,6) -- (3,6);
\draw[line width = 1.5pt, color = green!50!black] (3,6) -- (6,6) -- (6,0) -- (0,0) -- (0,6);

\draw (3,3) node {$\O$};
\draw (1.5,6) node[above]{$\color{red} \Gamma^D$};
\draw (6,3) node[right]{$\color{green!50!black} \Gamma^N$};
\end{tikzpicture}
\end{center}
\caption{Geometry of the two-dimensional domain}\label{fig:domain2D}
\end{figure}

We run our scheme on a Delaunay mesh made of 7374 triangles and 11167 edges generated with Gmsh for $3$ different initial concentration profiles. Two initial profiles $U^{0,(1)}$ and $U^{0,(2)}$ are chosen globally neutral: 
\be\label{eq:num.u0.1}
\begin{cases}
u_1^{0,(1)}(x) = 0.3 \times\mathbbm{1}_{(0,1/2)^2}(x), \\ 
u_2^{0,(1)}(x) = 0.3\times\mathbbm{1}_{(1/2,1)\times(0,1/2)}(x), \\
u_{{3}}^{0,(1)}(x) = 0.9\times\mathbbm{1}_{(1/2,1)^2}(x).
\end{cases}
\ee
and 
\be\label{eq:num.u0.2}
\begin{cases}
u_1^{0,(2)}(x) = 0.1\times u_1^{0,(1)}(x), \\
u_2^{0,(2)}(x) = 0.1\times u_{{2}}^{0,({1})}{(x)} + 0.9\times\mathbbm{1}_{(0,1)\times(1/2,1)}(x),\\
u_{{3}}^{0,(2)}(x) = 0.1\times u_{{3}}^{0,({1})}{(x)} + 0.9\times\mathbbm{1}_{(0,1)\times(0,1/2)}(x).
\end{cases}
\ee
A third initial configuration is chosen globally charged and constant in space: 
\be\label{eq:num.u0.3}
u_1^{0,(3)}(x) = 0.2, \quad 
u_2^{0,(3)}(x) = 0.2, \quad 
u_{{3}}^{0,(3)}(x) = 0.3.
\ee
The steady state corresponding to \eqref{eq:num.u0.3} is not constant w.r.t. space, as shown on Figure~\ref{fig:u3stat}. 
\begin{figure}[htb]
\centering
\begin{tabular}{cc}
\begin{tikzpicture}
\pgftext{\includegraphics[width=6cm]{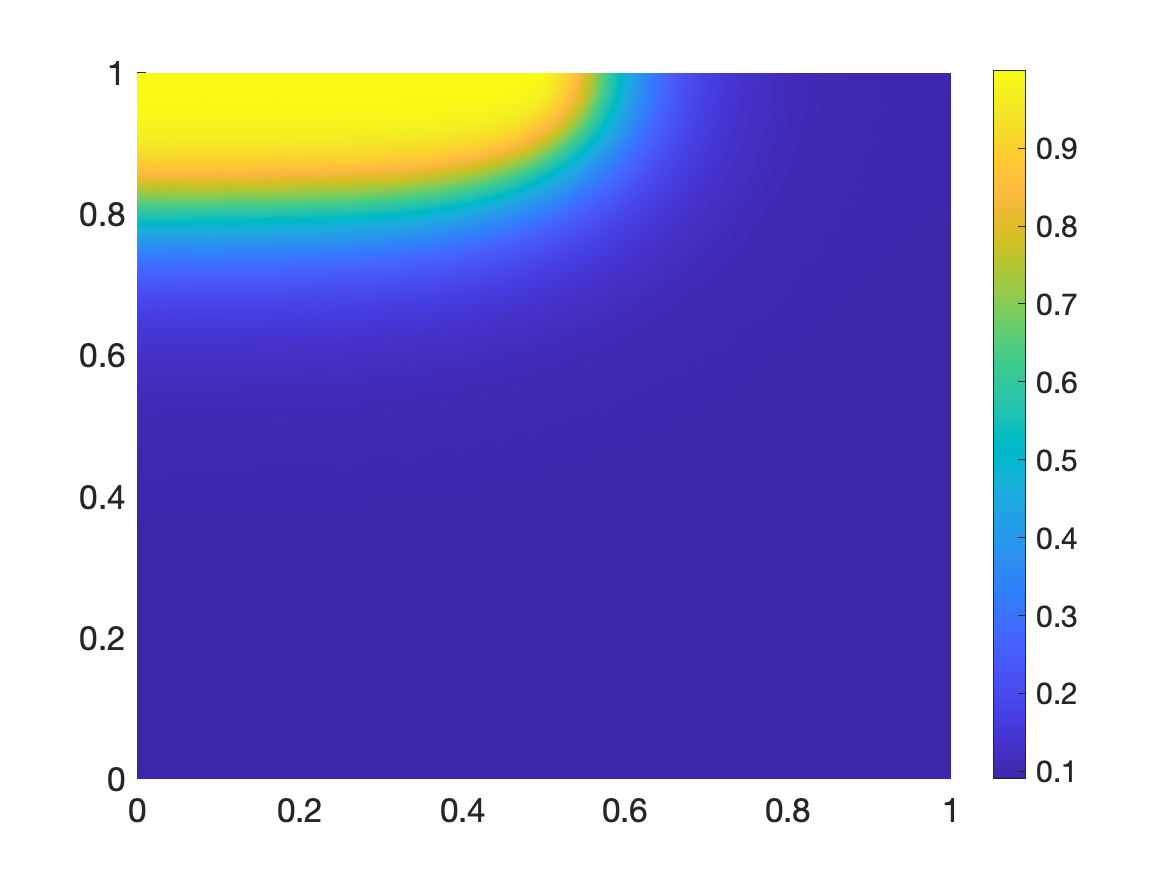}}

\draw[color = white, fill = white] (-2,2) -- (2,2) -- (2,2.5) -- (-2,2.5) -- (-2,2);

\end{tikzpicture}& 
\begin{tikzpicture}
\pgftext{
\includegraphics[width=6cm]{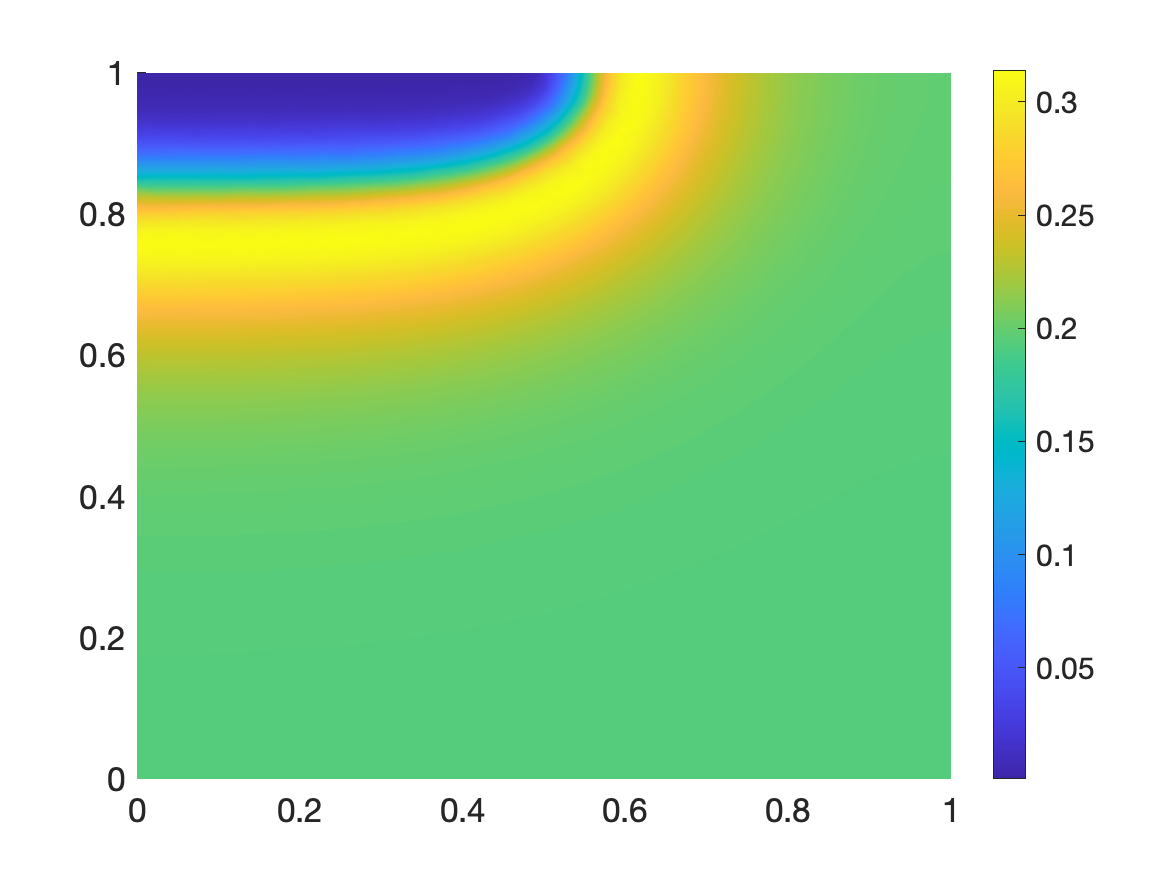}}

\draw[color = white, fill = white] (-2,2) -- (2,2) -- (2,2.5) -- (-2,2.5) -- (-2,2);
\end{tikzpicture}
\\
$u_1^\infty, \; z_1 = 2$ & 
$u_2^\infty, \; z_2 = 1$ \\
\begin{tikzpicture}
\pgftext{
\includegraphics[width=6cm]{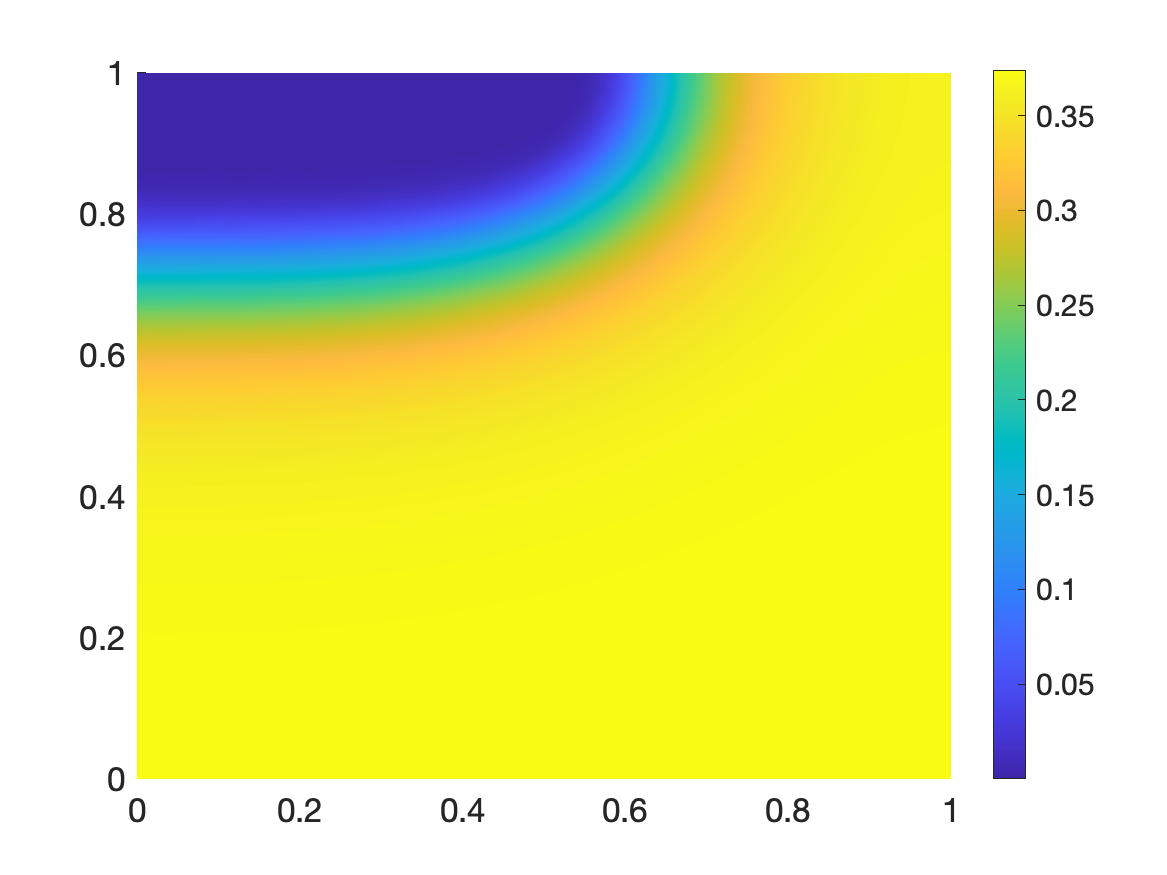}
}
\draw[color = white, fill = white] (-2,2) -- (2,2) -- (2,2.5) -- (-2,2.5) -- (-2,2);
\end{tikzpicture}
& 
\begin{tikzpicture}
\pgftext{
\includegraphics[width=6cm]{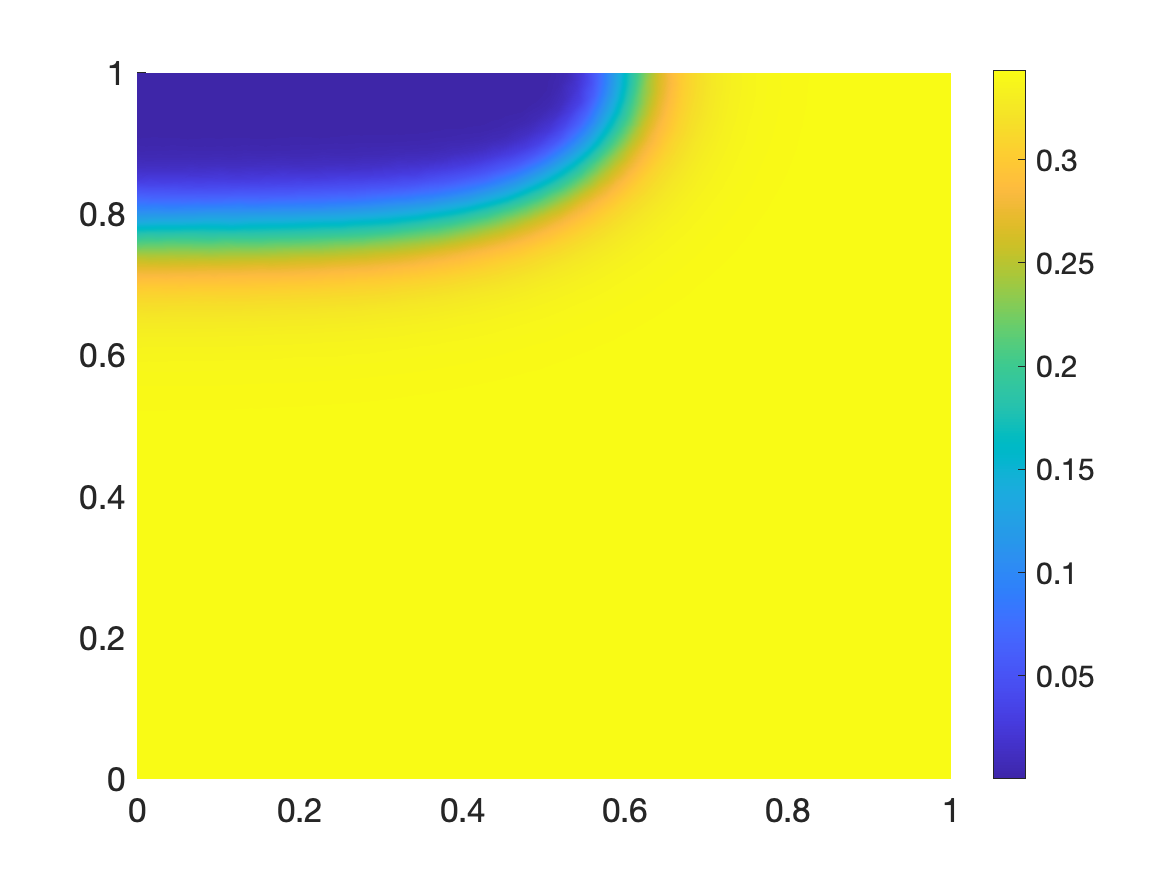} }
\draw[color = white, fill = white] (-2,2) -- (2,2) -- (2,2.5) -- (-2,2.5) -- (-2,2);
\end{tikzpicture}
\\
$u_3^\infty, \; z_3 = -1$ & 
$u_0^\infty, \; z_0 = 0$ 
\end{tabular}
\caption{Stationary concentration profiles corresponding to the initialization \eqref{eq:num.u0.3} for $\lambda^2 = 0.01 $}\label{fig:u3stat}
\end{figure}

We run our scheme with a constant time step $\tau = 10^{-4}$ and for two different Debye lengths until a final time $T=3$, and look for the evolution of the relative energy $\Hh_\Tt^n - \Hh_\Tt^\infty$ along-time. The relative energy is decaying for all the curves plotted on Figure~\ref{fig:EntRel}, but the velocity at which the decay occurs varies strongly depending on the Debye length and on the initial profile. In the electroneutral configurations, we mainly observe that higher quantities of $u_0$ yield a faster convergence. This is expected as the dissipation \eqref{eq:cont.dissip} contains a $u_0^2$ prefactor, in opposition to a mere $u_0$ for the classical Poisson--Nernst--Planck system. Then, still in the electroneutral regime, the smaller is $\lambda^2$ then the faster is the convergence towards the constant in space equilibrium. 
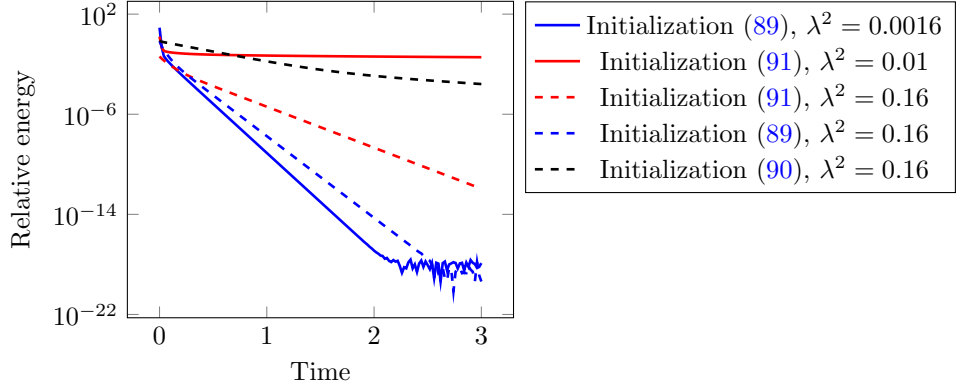
\begin{figure}[htb]
\begin{center}
\begin{tikzpicture}
	\begin{semilogyaxis}[
	xlabel=Time,
	ylabel=Relative energy,
	legend style={
		legend pos= outer north east,
	},
	width=0.45\linewidth]
	
	\addplot[color=blue] table[x=time, y=NRG_rel_abs] {results_long/TempsLong2D_EN1_14.txt};
	\addplot[color=red] table[x=time, y=NRG_rel_abs] {results_long/TempsLong_2D_C_14.txt};
	\addplot[color=red, dashed] table[x=time, y=NRG_rel_abs] {results_long/TempsLong_2D_C_14_lambdabig.txt};
	\addplot[color=blue, dashed] table[x=time, y=NRG_rel_abs] {results_long/TempsLong_2D_EN_14_lambdabig.txt};
	\addplot[color=black, dashed] table[x=time, y=NRG_rel_abs] {results_long/TempsLong2D_EN2_14_lambdabig.txt};
	\addlegendentry{Initialization \eqref{eq:num.u0.1}, $\lambda^2 = 0.0016$}
	\addlegendentry{Initialization \eqref{eq:num.u0.3}, $\lambda^2 = 0.01$}
	\addlegendentry{Initialization \eqref{eq:num.u0.3}, $\lambda^2 = 0.16$}
	\addlegendentry{Initialization \eqref{eq:num.u0.1}, $\lambda^2 = 0.16$}
	\addlegendentry{Initialization \eqref{eq:num.u0.2}, $\lambda^2 = 0.16$}
	\end{semilogyaxis}
\end{tikzpicture}
\end{center}
\caption{Convergence towards the steady long-time behavior in terms of relative energy $\Hh_\Tt^{n} - \Hh_\Tt^\infty$.}
\label{fig:EntRel}
\end{figure}
The trend is different for globally charged profiles. For the choice~\eqref{eq:num.u0.3} of the initial profile, the convergence becomes extremely slow when $\lambda^2$ becomes small. This is due to the fact that, in some areas, the concentration of $u_0$ becomes extremely small. The equilibration between the other phases is then extremely slow in these regions. To illustrate this difficulty, we depict on Figure~\ref{fig:u3_final} the concentration profiles at $T=3$. The solvent concentration $u_0$ is already in good agreement with the equilibrium profile depicted on Figure~\ref{fig:u3stat} (see also Figure~\ref{fig:u3error} for a representation of the difference between $U_{\Tt,\tau}(T)$ and $U_\Tt^\infty$), but this is not the case of the other concentrations, the evolution of which being stuck in the zones where the concentration of $u_0$ is small. Obtaining a fine quantitative understanding of the long-time behavior of the model should be investigated in future works. 
\begin{figure}[htb]
\centering
\begin{tabular}{cc}
\begin{tikzpicture}
\pgftext{
\includegraphics[width=6cm]{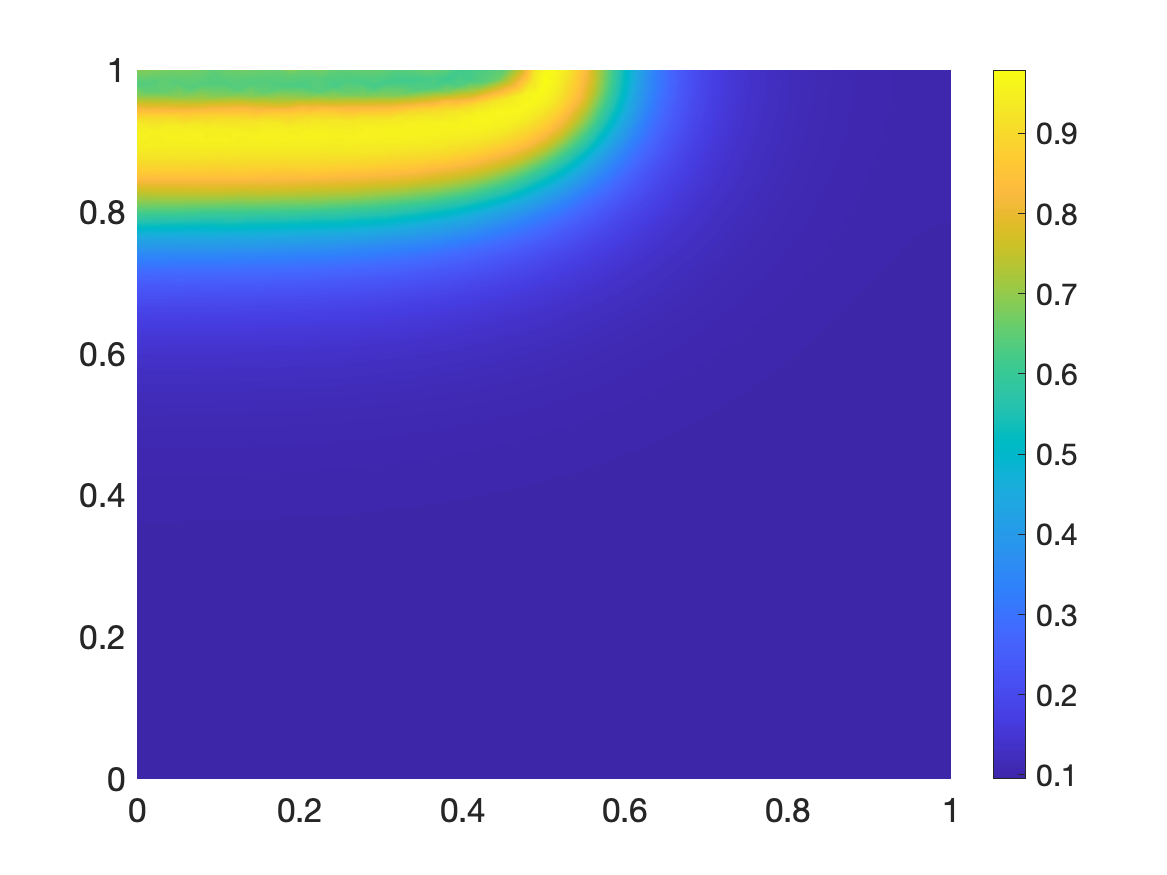}}
\draw[color = white, fill = white] (-2,2) -- (2,2) -- (2,2.5) -- (-2,2.5) -- (-2,2);
\end{tikzpicture}& 
\begin{tikzpicture}
\pgftext{
\includegraphics[width=6cm]{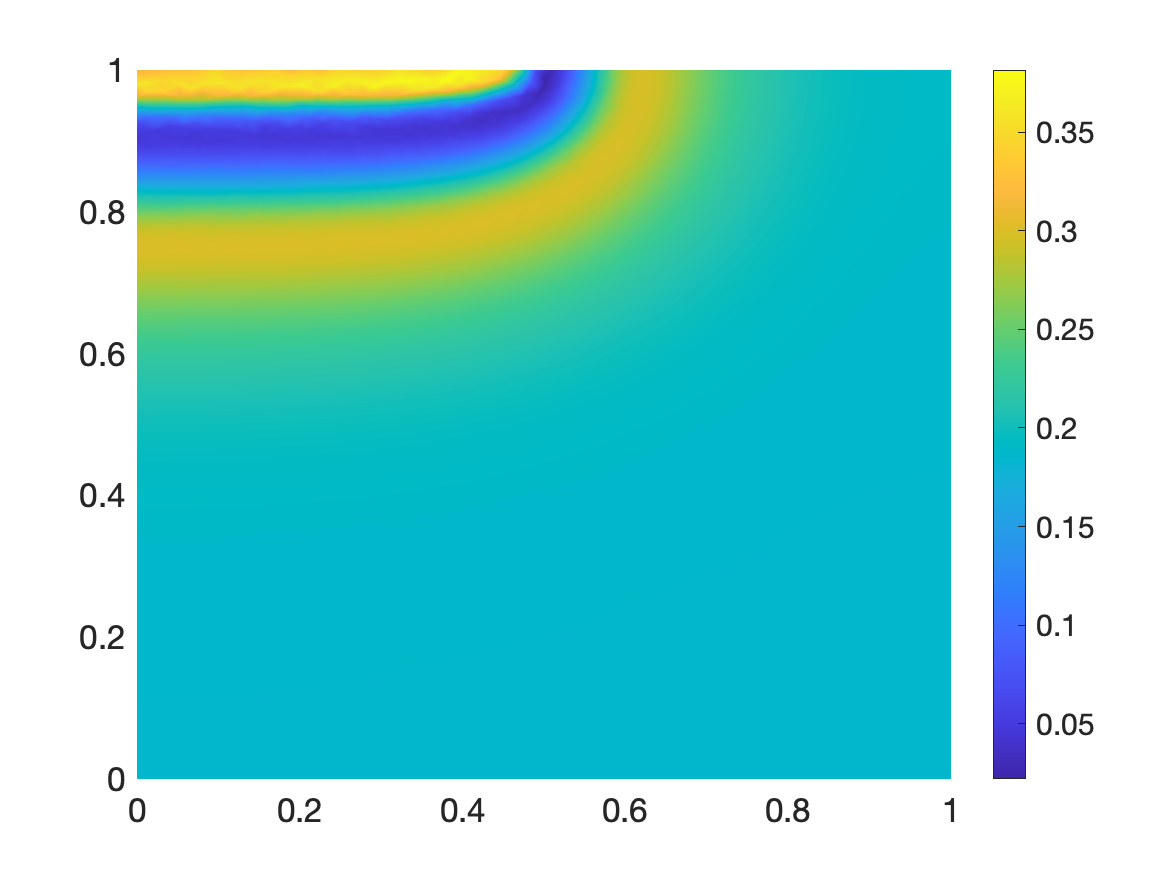}}
\draw[color = white, fill = white] (-2,2) -- (2,2) -- (2,2.5) -- (-2,2.5) -- (-2,2);
\end{tikzpicture} \\
$u_1(T), \; z_1 = 2$ & 
$u_2(T), \; z_2 = 1$ \\
\begin{tikzpicture}
\pgftext{
\includegraphics[width=6cm]{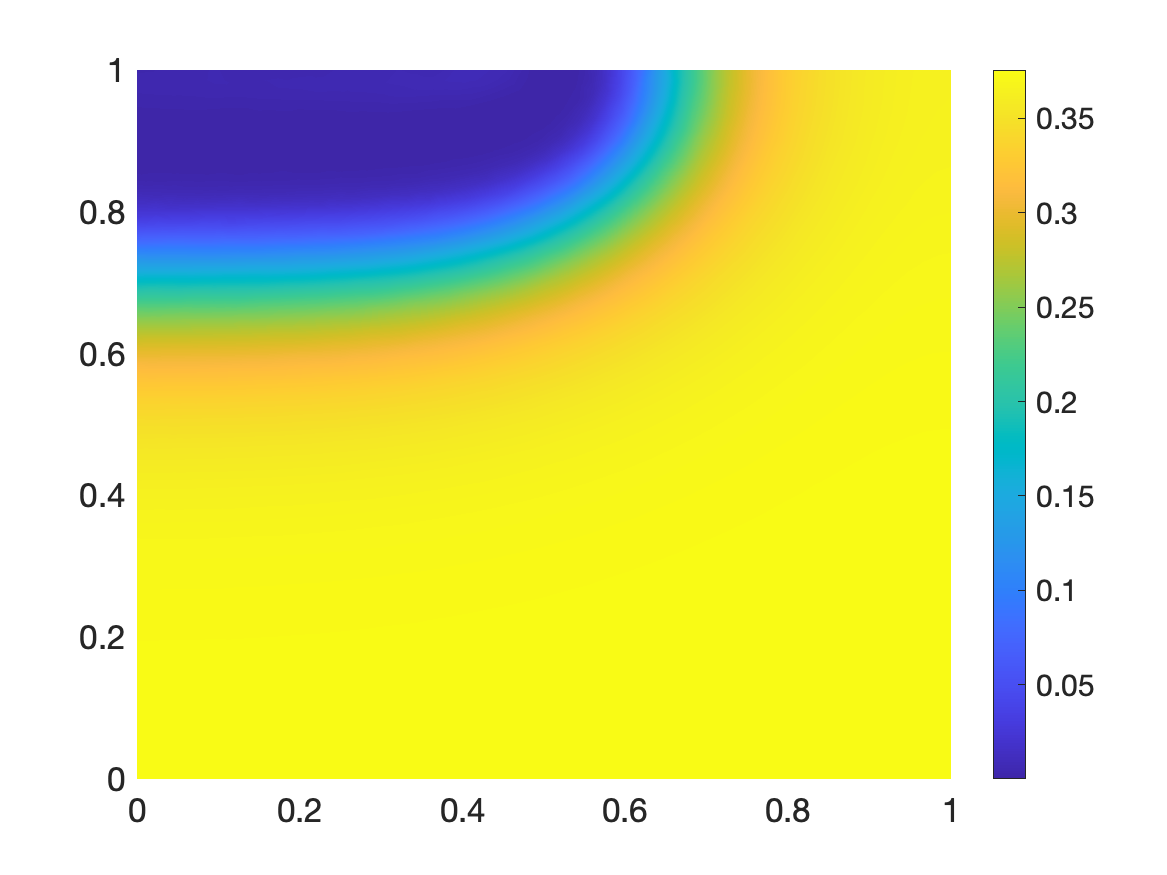}}
\draw[color = white, fill = white] (-2,2) -- (2,2) -- (2,2.5) -- (-2,2.5) -- (-2,2);
\end{tikzpicture} & 
\begin{tikzpicture}
\pgftext{
\includegraphics[width=6cm]{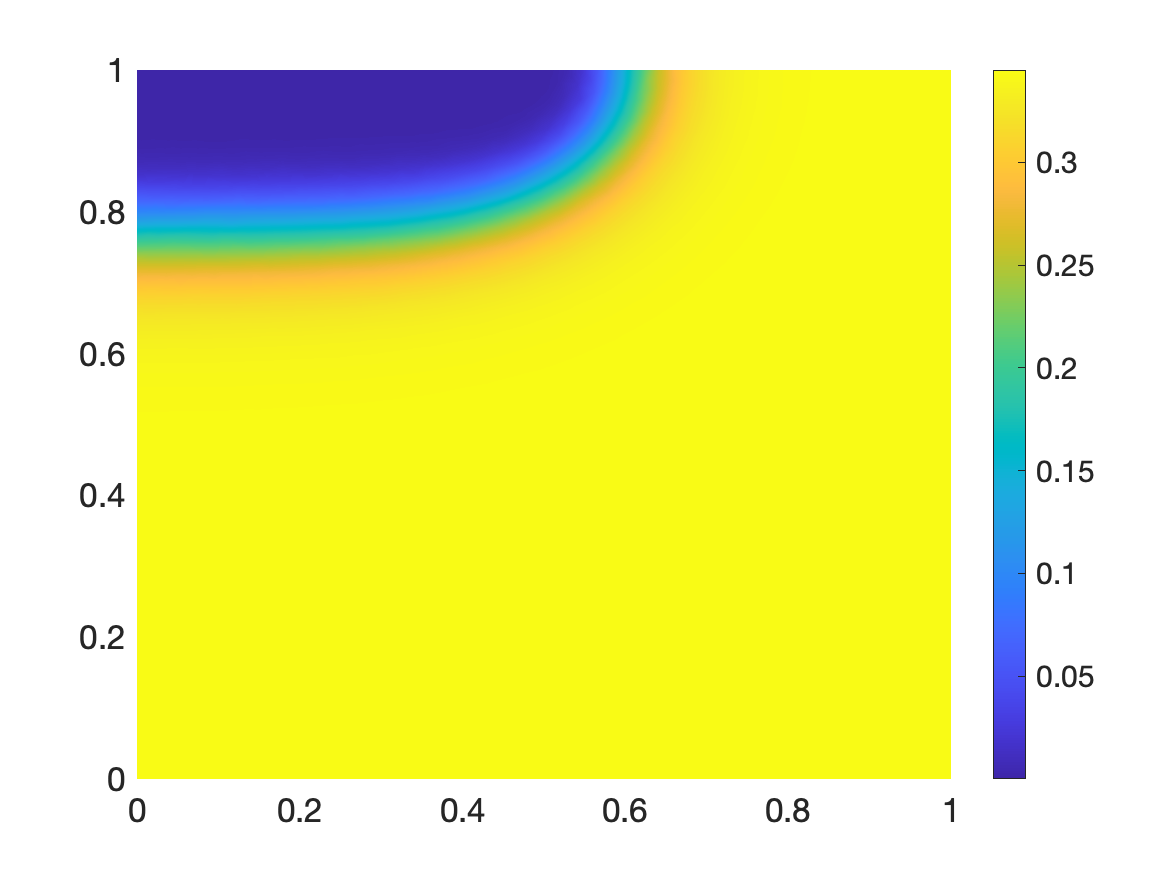}}
\draw[color = white, fill = white] (-2,2) -- (2,2) -- (2,2.5) -- (-2,2.5) -- (-2,2);
\end{tikzpicture} 
\\
$u_3(T), \; z_3 = -1$ & 
$u_0(T), \; z_0 = 0$ 
\end{tabular}
\caption{Concentration profiles at final time $T=3$ corresponding to the initialization \eqref{eq:num.u0.3} for $\lambda^2 = 0.01$.}\label{fig:u3_final}
\end{figure}

\begin{figure}[htb]
\centering
\begin{tabular}{cc}
\begin{tikzpicture}
\pgftext{
\includegraphics[width=6cm]{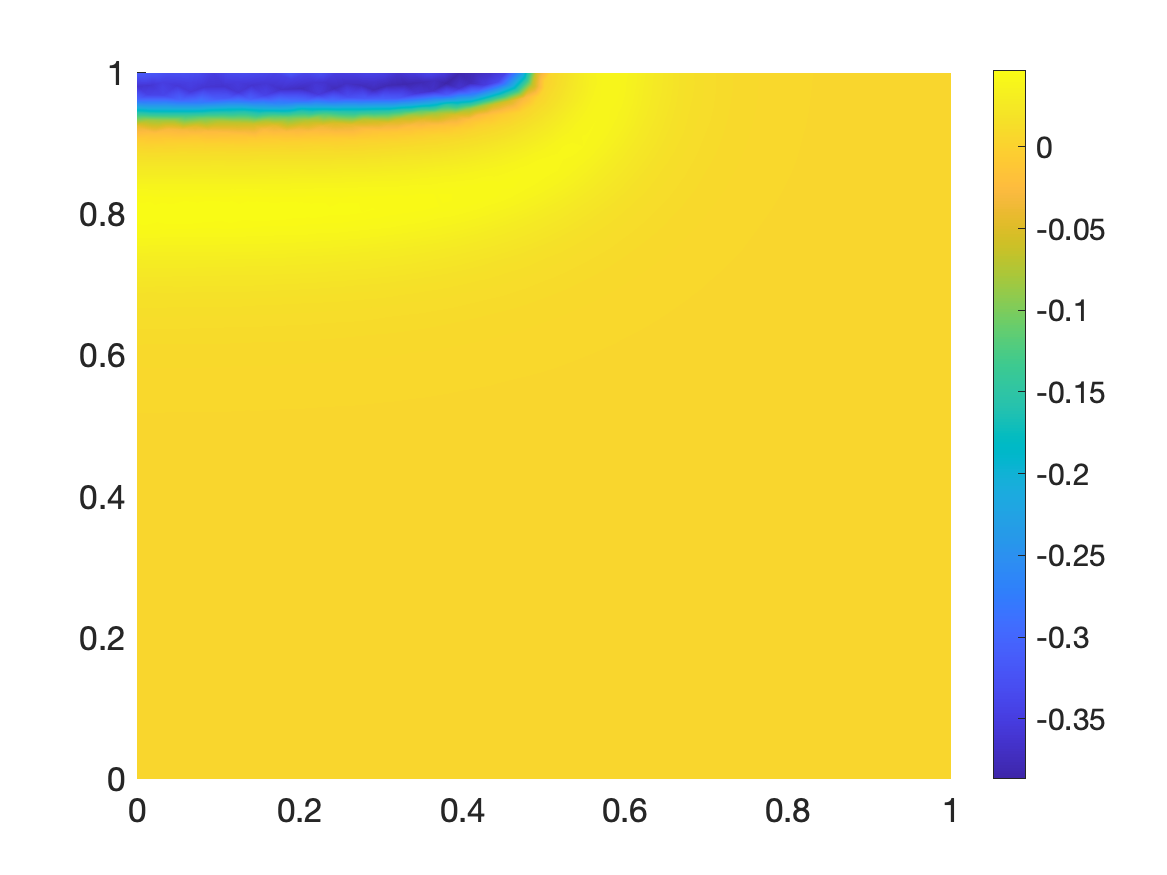}}
\draw[color = white, fill = white] (-2,2) -- (2,2) -- (2,2.5) -- (-2,2.5) -- (-2,2);
\end{tikzpicture}
& 
\begin{tikzpicture}
\pgftext{
\includegraphics[width=6cm]{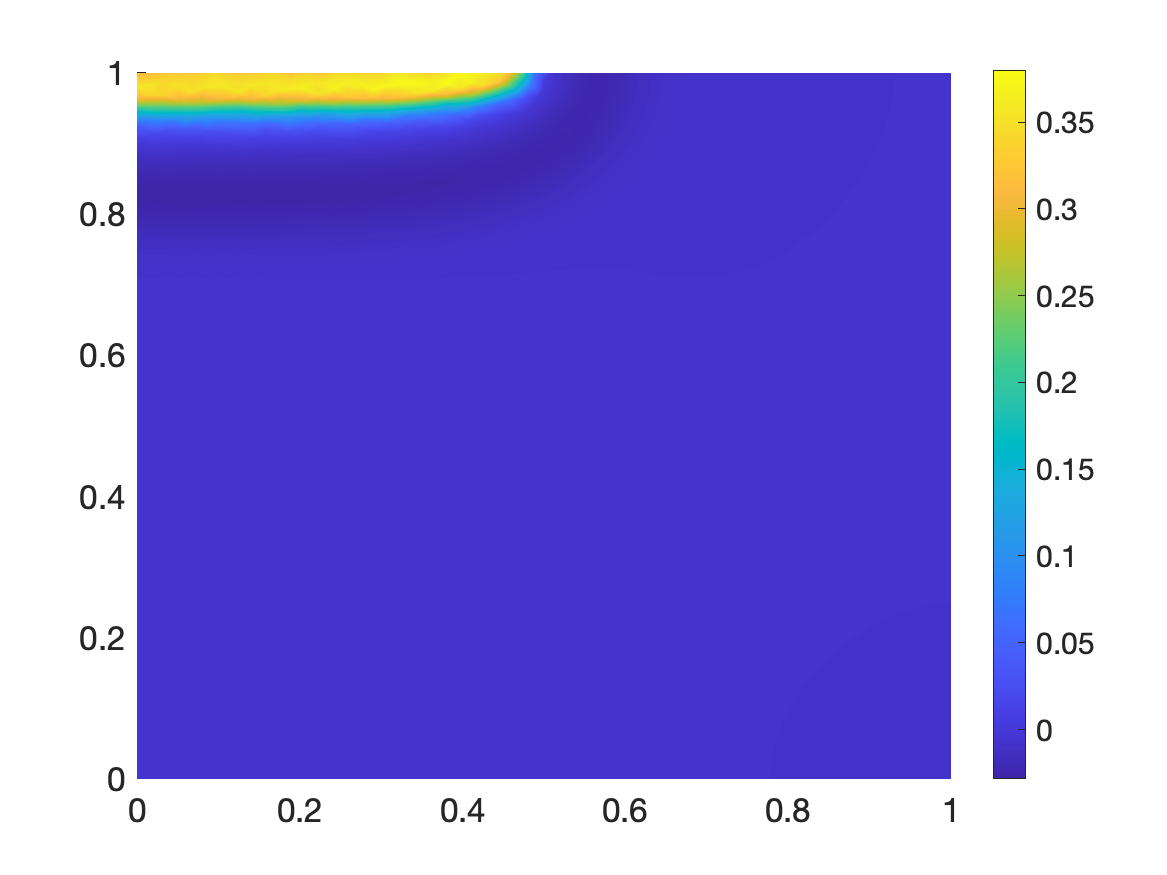}}
\draw[color = white, fill = white] (-2,2) -- (2,2) -- (2,2.5) -- (-2,2.5) -- (-2,2);
\end{tikzpicture} \\
$u_1(T) - u_1^\infty, \; z_1 = 2$ &
$u_2(T) - u_2^\infty, \; z_2 = 1$ \\
\begin{tikzpicture}
\pgftext{
\includegraphics[width=6cm]{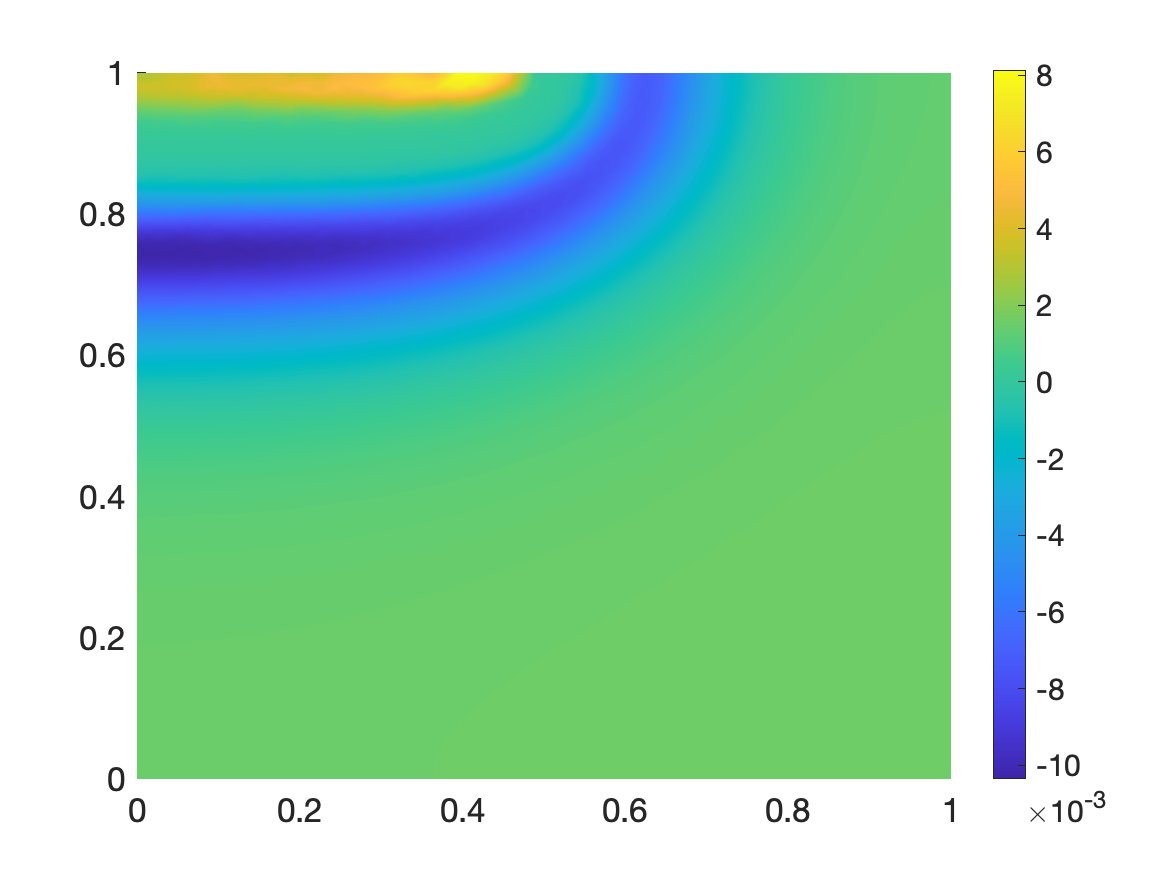}}
\draw[color = white, fill = white] (-2,2) -- (2,2) -- (2,2.5) -- (-2,2.5) -- (-2,2);
\end{tikzpicture} & 
\begin{tikzpicture}
\pgftext{
\includegraphics[width=6cm]{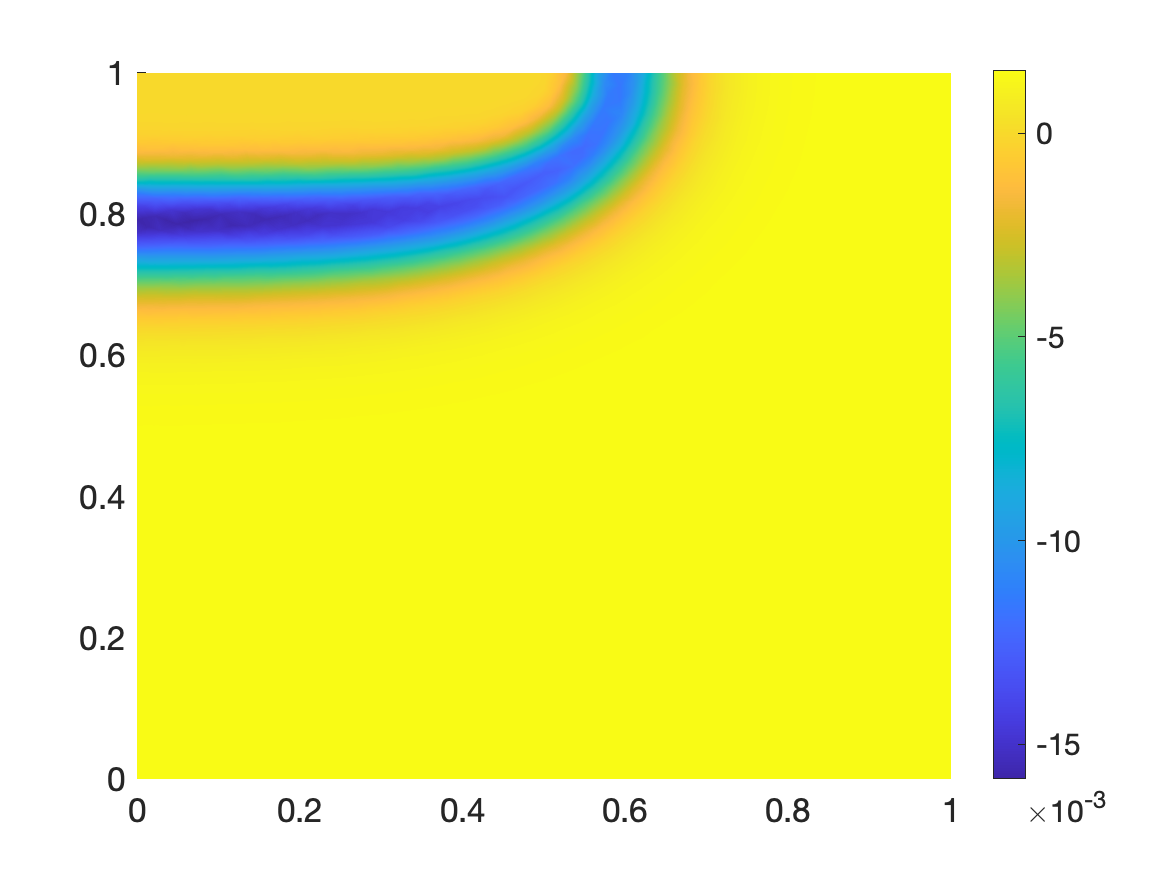}}
\draw[color = white, fill = white] (-2,2) -- (2,2) -- (2,2.5) -- (-2,2.5) -- (-2,2);
\end{tikzpicture} 
\\
$u_3(T) - u_3^\infty, \; z_3 = -1$ &
$u_0(T) - u_0^\infty, \; z_0 = 0$
\end{tabular}
\caption{Difference between the final concentration profiles of Figure~\ref{fig:u3_final} and the stationary ones, cf.  Figure~\ref{fig:u3stat}\eqref{eq:num.u0.3} for $\lambda^2 = 0.01$.}\label{fig:u3error}
\end{figure}

\medskip

\noindent
\textbf{Acknowledgement.} CC and MH acknowledge support from the LabEx CEMPI (ANR-11-LABX-0007) and the ministries of Europe and Foreign Affairs (MEAE) and Higher Education, Research and Innovation (MESRI) through PHC Amadeus 46397PA. AM acknowledges support from the multilateral project of the Austrian Agency for International Co-operation in Education and Research (OeAD), grant FR 01/2021, as well as partial support from the Austrian Science Fund (FWF), grant DOI 10.55776/P33010
and 10.55776/F65. This work has received funding from the European Research Council (ERC) under the European Union’s Horizon 2020 research and innovation programme, ERC Advanced Grant NEUROMORPH, no. 101018153. CC also acknowledges support from the European Union's Horizon 2020 research 
and innovation programme under grant agreement No 847593 (EURAD program, WP DONUT) and from the COMODO (ANR-19-CE46-0002) project.

\end{document}